\documentclass[11pt]{amsart}

\usepackage[margin=1in]{geometry}  
\usepackage{graphicx}              
\usepackage{amsmath}               
\usepackage{amsfonts}              
\usepackage{amsthm}                
\usepackage{amsmath, amssymb}
\usepackage{oubraces}
\usepackage{tikz-cd}
\usepackage[all]{xy}
\usepackage{enumerate}
\usepackage{graphicx}
\usepackage{color}

\newtheorem{thm}{Theorem}[section]
\newtheorem{lem}[thm]{Lemma}
\newtheorem{prop}[thm]{Proposition}
\newtheorem{cor}[thm]{Corollary}

\theoremstyle{definition}
\newtheorem{remark}[thm]{Remark}

\theoremstyle{definition}
\newtheorem{defn}[thm]{Definition}

\theoremstyle{definition}
\newtheorem{ex}[thm]{Example}

\newcommand{\Str}{\operatorname{Str}}

\begin{document}

\begin{abstract}
Maximal green sequences are important objects in representation theory, cluster algebras, and string theory.   The two fundamental questions about maximal green sequences are whether a given algebra admits such sequences and, if so, does it admit only finitely many.  We study maximal green sequences in the case of string algebras and give sufficient conditions on the algebra that ensure an affirmative answer to these questions. 
\end{abstract}

\title{Maximal Green Sequences for String Algebras}
\author{Alexander Garver}
\address{Department of Mathematics and Statistics,
Carleton College}
\email{alexander.garver@gmail.com}

\author{Khrystyna Serhiyenko}
\address{Department of Mathematics,
University of Kentucky}
\email{khrystyna.serhiyenko@uky.edu}

\maketitle
\section{Introduction}

A maximal green sequence of a \textit{2-acyclic} quiver (i.e., a quiver without loops or 2-cycles) is a distinguished sequence of mutations, in the sense of Fomin and Zelevinsky \cite{fomin2002cluster}, of the quiver. Maximal green sequences were introduced in \cite{keller2011cluster} to provide explicit formulas for the refined Donaldson--Thomas invariants of Kontsevich and Soibelman \cite{kontsevich2008stability}. They are also important in string theory \cite{alim2013bps}, representation theory \cite{brustle2014maximal, brustle2013ordered}, and cluster algebras \cite{gross2018canonical}. 

In particular, in \cite{brustle2013ordered}, it is shown that maximal green sequences of a 2-acyclic quiver $Q$ are in bijection with finite maximal chains in the lattice of torsion classes in the module category of the Jacobian algebra associated with $Q$, in the sense of \cite{derksen2008quivers}, when the Jacobian algebra of $Q$ is finite dimensional. This characterization allows one to define maximal green sequences more generally from the initial data of a finite dimensional algebra $\Lambda = KQ/I$. 

There are two fundamental questions about the maximal green sequences of $\Lambda$:
\begin{enumerate}
\item Does the algebra $\Lambda$ admit a maximal green sequence?
\item If so, does it admit only finitely many?
\end{enumerate}

The first question has been answered in the affirmative when $\Lambda = KQ$ for $Q$ an acyclic quiver \cite[Lemma 2.20]{brustle2014maximal} and when $\Lambda = KQ/I$ is a Jacobian algebra and $Q$ is almost any quiver of finite mutation type \cite[Theorem 1.3]{mills2017maximal}. The second question has been answered in the affirmative in the case when $\Lambda = KQ$ for 
{an acyclic quiver $Q$ that comes from an orientation of Dynkin or an extended Dynkin diagram or that contains at most three vertices} \cite[Theorems 4.4, 5.2, and 5.4 and Lemma 5.1]{brustle2014maximal}. Futhermore, if $\Lambda = KQ/I$ is a Jacobian algebra where $Q$ is obtained by a finite sequence of mutations of a quiver that is an acyclic orientation of a Dynkin diagram or an extended Dynkin diagram, then $\Lambda$ has only finitely many maximal green sequences \cite[Theorem 2]{brustle2017semi}. We remark that the result cited in the previous sentence holds in the greater generality where $Q$ is assumed to be a valued quiver, but we do not work in such generality in this paper.

{ The goal of this paper is to find sufficient conditions on the algebras so that they admit a maximal green sequence or admit only finitely many maximal green sequences.  Here we study the case of string algebras.} In Corollary~\ref{dom_str_cor}, we show that any domestic string algebra admits at most finitely many maximal green sequences. In Corollary~\ref{deg_3_cor}, we show that any string algebra with the property that each vertex of its quiver has degree at most three admits a maximal green sequence. We also show in Theorem~\ref{dom_gentle_thm} any domestic gentle algebra admits a maximal green sequence.
{Thus, we present new examples of algebras for which the two questions have a positive answer.}

A wide range of techniques have been used to study maximal green sequences, including the combinatorics of surface triangulations, semi-invariant pictures, stability conditions, and torsion classes. The reason that we restrict to string algebras is because the indecomposable modules over a string algebra and morphisms between the indecomposable modules are completely classified; any indecomposable module is either a string or a band module. Such modules are determined by certain walks in the quiver of the algebra known as strings and bands. This classification allows one to make use of yet another classification of the maximal green sequences of an algebra called complete forward hom-orthogonal sequences, in the language of Igusa \cite{igusa2019maximal}. 

A complete forward hom-orthogonal sequence for $\Lambda$ is a sequence of bricks $M_1,\ldots, M_k$ of $\Lambda$ where $\text{Hom}_\Lambda(M_i, M_j) = 0$ for all $i < j$ that cannot be refined in such a way that preserves this condition. Igusa proved that such sequences are in bijection with maximal green sequences of $\Lambda = KQ/I$ when $\Lambda$ is the Jacobian algebra of a quiver that is mutation-equivalent to an acyclic orientation of a Dynkin diagram \cite[Corollary 2.14]{igusa2019maximal}. Demonet and Keller then proved that this bijection stills holds when $\Lambda$ is a finite dimensional algebra \cite[Theorem A.3]{demonet2019survey}. Complete forward hom-orthogonal sequences were also used in \cite{nasr2019maximal} to understand some of the maximal-length maximal green sequences for Jacobian algebras of quivers that are mutation-equivalent to acyclic orientations of Dynkin diagrams.

The paper is organized as follows. In Section~\ref{sec_prelim}, we review the definitions of maximal green sequences, string and band modules, and string algebras. Since the bricks of an algebra are essential to understand the complete forward hom-orthogonal sequences of the algebra, we focus on the properties of the bricks of string algebras in Section~\ref{sec:bricks}. In Section~\ref{sec:mgs}, we apply our results on the bricks of a string algebra from Section~\ref{sec:bricks} to obtain our main results. 


%
%

\section{Preliminaries}\label{sec_prelim}

\subsection{Maximal green sequences}\label{sec:MGS} We present the representation-theoretic formulation of maximal green sequences, based on the work of \cite{brustle2017stability, demonet2019survey, demonet2017lattice, igusa2019maximal}. Let $\Lambda= KQ/I$ be a finite dimensional algebra over an algebraically closed field $K$, and let $\Lambda$-mod denote the category of all finitely generated modules over $\Lambda$.  


A full subcategory $\mathcal{T} \subseteq \Lambda$-mod is called a \textit{torsion class} if it is closed under extension and quotients. By partially ordering the torsion classes of $\Lambda$-mod by containment, the collection of all torsion classes forms a complete lattice. We let $\textsf{tors}\Lambda$ denote the lattice of torsion classes of $\Lambda$-mod. Br\"ustle, Smith, and Treffinger define a \textit{maximal green sequence} for $\Lambda$ as a finite maximal chain in the lattice of torsion classes of $\Lambda$-mod \cite{brustle2017stability}.  

We say a module $M \in \Lambda$-mod is a \textit{brick} if $\text{End}_{\Lambda}(M)$ is a division algebra. As introduced in \cite[Theorem 3.3]{demonet2017lattice}, for each \textit{covering relation} $\mathcal{T}_1 \lessdot \mathcal{T}_2$ in $\textsf{tors}\Lambda$ (i.e., $\mathcal{T}_1 \subsetneq \mathcal{T}_2$ in $\textsf{tors}\Lambda$ and there does not exist another torsion class $\mathcal{T} \in \textsf{tors}\Lambda$ satisfying $\mathcal{T}_1 \subsetneq \mathcal{T} \subsetneq \mathcal{T}_2$) there is a unique brick $M \in \Lambda$-mod up to isomorphism that satisfies $M \in \mathcal{T}_1$ and $\text{Hom}_\Lambda(T_2, M) = 0$ for all $T_2 \in \mathcal{T}_2$. Consequently, any finite length maximal chain $\mathcal{T}_1 \subseteq \mathcal{T}_2 \subseteq \cdots \subseteq \mathcal{T}_k$ in $\textsf{tors}\Lambda$ gives rise to a sequence of isomorphism classes of bricks $M_1, M_2, \ldots, M_{k-1}$ where $M_i$ is the unique brick associated with the covering relation $\mathcal{T}_i \lessdot \mathcal{T}_{i+1}$. Demonet and Keller showed that this map from finite length maximal chains to sequences of bricks is a bijection in the sense of the following theorem.


\begin{thm}\cite[Theorem A.3]{demonet2019survey}\label{MGS_thm}
A sequence of bricks $M_1, M_2, \ldots, M_{k}$ arises from a maximal chain of torsion classes of $\Lambda$ if and only if $\text{Hom}_{\Lambda}(M_i, M_j) = 0$ for all $i < j$ and this sequence cannot be refined in a way that preserves this condition. In particular, the maximal green sequences for $\Lambda$ are parameterized by finite sequences of bricks $M_1, M_2, \ldots, M_k$ satisfying the latter conditions.
\end{thm}

The condition on sequences of bricks in the above theorem was first written down by Igusa in \cite[Definition 2.2]{igusa2019maximal}. He referred to such sequences of bricks as \textit{complete forward hom-orthogonal (FHO) sequences}, and he showed that when $\Lambda$ is a cluster-tilted algebra of finite representation type, its complete FHO sequences parameterize the maximal green sequences of $\Lambda$ \cite[Corollary 2.14]{igusa2019maximal}. 

For convenience, we recall two additional terms from Igusa's work. We say a sequence of bricks $M_1, M_2, \ldots$ is a \textit{weakly FHO sequence} if $\text{Hom}_{\Lambda}(M_i, M_j) = 0$ for any $i < j$. In light of Theorem~\ref{MGS_thm}, we will use the terms maximal green sequence and a complete forward hom-orthogonal sequence interchangeably.  It will be clear from the context whether in a given instant we are referring to a sequence of bricks or a sequence of torsion classes.

\subsection{String and band modules} Let $Q=(Q_0, Q_1)$ be a finite quiver where $Q_0$ denotes the set of vertices and $Q_1$ denotes the set of arrows in $Q$.   Given an arrow $\gamma \in Q_1$, its starting and ending vertices are denoted by $s(\gamma), t(\gamma)$ respectively, where $s(\gamma) \xrightarrow{\gamma} t(\gamma)$.   Let $\Lambda=KQ/I$ be a finite dimensional algebra over an algebraically closed field $K$. We formally define $Q_1^{-1}$ to be the set of \textit{inverse arrows} of $Q$. Elements of $Q_1^{-1}$ are denoted by $\gamma^{-1}$, for $\gamma \in Q_1$, and $s(\gamma^{-1}):=t(\gamma)$ and $t(\gamma^{-1}):=s(\gamma)$.
A \emph{string} in $\Lambda$ of length $d\geq 1$ is a word $w = \gamma_1^{\varepsilon_1}\cdots\gamma_d^{\varepsilon_d}$ in the alphabet $Q_1 \sqcup Q_1^{-1}$ with $\varepsilon_i \in \{\pm 1\}$, for all $i \in \{1,2,\cdots,d\}$, which satisfies the following conditions:

\begin{enumerate}
\item[(1)] $s(\gamma_{i+1}^{\varepsilon_{i+1}})=t(\gamma_i^{\varepsilon_i})$ and $ \gamma_{i+1}^{\varepsilon_{i+1}} \neq \gamma_i^{-\varepsilon_i}$, for all $i\in \{1,\cdots,d-1 \}$;
\item[(2)] $w$ and also $w^{-1} := \gamma_d^{-\varepsilon_d}\cdots\gamma_1^{-\varepsilon_1}$ do not contain a subpath in $I$.
\end{enumerate}


We refer to the symbols $\gamma_i^1$ and $\gamma_j^{-1}$ appearing in some string $w$ as \textit{arrows} and \textit{inverse arrows} of $w$. In the case that $\gamma_i^{\varepsilon_i}$ has $\varepsilon_i = 1$, we will simply write $\gamma_i$. We also associate a string of length zero to every vertex $i \in Q_0.$ We denote this string by $e_i$. We let $\Str(\Lambda)$ denote the set of strings in $\Lambda$ up to the equivalence relation where a string $w$ is identified with $w^{-1}$; we write this equivalence as $w \sim w^{-1}$. 

{We say that $v$ is a \textit{substring} of $w$ if $v = \gamma_i^{\varepsilon_i}\cdots \gamma_j^{\varepsilon_j}$ for some $1 \le i \le j \le d$ or if $v = e_i$ for some vertex $i \in Q_0$ on which $w$ is supported.}

We say $w$ \textit{starts} at $s(w)=s(\gamma_1^{\varepsilon_1})$ and \textit{ends} at $t(w)=t(\gamma_d^{\varepsilon_d})$. We say that a string $w$ of positive length is a \emph{cyclic} string if $s(w) = t(w)$. If $w$ is a cyclic string, it is called a \emph{band} if $w^m$ is a string for each $m \in \mathbb{Z}_{\geq 1}$ and $w$ is \textit{primitive} (i.e., it is not a power of a string of strictly smaller length).  It is important to note that, by convention, when we write $w^m$ for some string $w$ and $m \ge 2$, we require that $w$ be a string of length at least 1.  

\begin{ex}\label{example1} Consider the algebra $\Lambda = KQ/\langle \alpha^2, \gamma^2 \rangle$ where $Q$ is the following quiver
$$\xymatrix@C=20pt@R=20pt{
1\ar[rr]^{\beta} \ar@(dl,ul)^\alpha &&2 \ar@(dr, ur)_\gamma}.$$
Some examples of strings in this algebra are $e_1, e_2, \alpha^1, \beta^1, \gamma^1, \beta^{1}\gamma^{-1}\beta^{-1}$, and $\alpha^{1}\beta^1\gamma^{-1}\beta^{-1}$. Among these only $\alpha^{1}\beta^1\gamma^{-1}\beta^{-1}$ is a band.
\end{ex}


Let $w =  \gamma_1^{\varepsilon_1}\cdots \gamma_d^{\varepsilon_d}$ be an element of $\Str(\Lambda)$. We can express the walk on $Q$ determined by $w$ as the sequence \[\xymatrix{x_{1} \ar@{-}^{\gamma_1}[r] & x_2 \ar@{-}^{\gamma_{2}}[r] & \cdots & x_{d+1} \ar@{-}_{\gamma_d}[l]}\] where $x_1,\ldots, x_{d+1}$ are the vertices of $Q$ visited by $w$, a priori each one may be visited multiple times. The orientation of arrows is suppressed in this notation. The \textit{string module} defined by $w$ is the quiver representation $ M(w) := ((V_i)_{i \in Q_0}, (\varphi_\alpha)_{\alpha\in Q_1})$ with vector spaces given by
$$\begin{array}{cccccccccccc}
V_i & := & \left\{\begin{array}{ccl} \displaystyle\bigoplus_{j: x_j = i}Kx_j &: & \text{if } i = x_j \text{ for some } j \in \{1,\ldots, d+1\}\\ 0 & : & \text{otherwise} \end{array}\right.
\end{array}$$ for each $i \in Q_0$ and with linear transformations given by
$$\begin{array}{cccccccccccc}
\varphi_\alpha(x_k) & := & \left\{\begin{array}{lcl} x_{k-1} &: & \text{if } \alpha = \gamma_{k-1} \text{ and } \varepsilon_k = -1\\ x_{k+1} &: & \text{if } \alpha = \gamma_{k} \text{ and } \varepsilon_k = 1\\  0 & : & \text{otherwise} \end{array}\right.
\end{array}$$ for each $\alpha \in Q_0$. We see that $\dim_K(V_i) = |\{j \in \{1,\ldots, d+1\} \mid \ x_j = i\}|$ for any $i \in Q_0$, and for any string $w$, one has $M(w) \simeq M(w^{-1})$ as $\Lambda$-modules.

If $w =  \gamma_1^{\varepsilon_1}\cdots \gamma_d^{\varepsilon_d}$ is a band in $\Str(\Lambda)$, it gives rise to a band module, as well as to a string module. This \textit{band module}, denoted $M(w,\lambda, k) = ((V_i)_{i \in Q_0}, (\varphi_\alpha)_{\alpha\in Q_1})$ with $\lambda \in K^*$ and $k \in \mathbb{N}_{>0}$, is defined as follows:
\begin{itemize}
\item for each $i \in Q_0$, we have $V_i = K^k$;
\item the linear maps are given by $$\begin{array}{cccccccccccc}
\varphi_\alpha & := & \left\{\begin{array}{lcl} I_k &: & \text{if } \alpha = \gamma_{i}  \text{ for some } i \in \{1,\ldots, d-1\}\\ J_k(\lambda^{\varepsilon_d}) &: & \text{if } \alpha = \gamma_{d} \\  0 & : & \text{otherwise} \end{array}\right.\end{array}$$ where $I_k$ is the $k\times k$ identity matrix and $J_k(\lambda^{\varepsilon_d})$ is the $k \times k$ lower-triangular Jordan matrix with eigenvalue $\lambda^{\varepsilon_d}$.
\end{itemize}
From the definition, one verifies that $M(w,\lambda,k) \simeq M(w^{-1},\lambda^{-1}, k)$ and $M(w,\lambda,k) \simeq M(w^\prime, \lambda, k)$ where $w^\prime$ is any band obtained by applying a cyclic permutation $\sigma \in \mathfrak{S}_d$ to $w$ in the sense that $w^\prime = \gamma_{\sigma(1)}^{\varepsilon_{\sigma(1)}}\cdots \gamma_{\sigma(d)}^{\varepsilon_{\sigma(d)}}$. In this case we write $w\sim w'$.  The following definition identifies particularly important types of bands for our study of maximal green sequences.

\begin{defn}
We say a band $w$ is \textit{minimal} if for all bands $w' \sim w$ there does not exist another band $v$ such that $v^k$ is a substring of $w'$ for some $k\geq 2$. 
\end{defn}



\begin{ex}\label{example2}
Consider the algebra $\Lambda = KQ/\langle \beta_1\beta_2, \gamma_2\gamma_1\rangle$ where $Q$ is the following quiver 
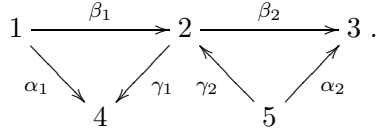
\[\xymatrix@C=20pt@R=20pt{
1\ar[rr]^{\beta_1} \ar[dr]_{\alpha_1} & &2 \ar[dl]^{\gamma_1} \ar[rr]^{\beta_2} & & 3 \\
&4 & & 5 \ar[ul]^{\gamma_2} \ar[ur]_{\alpha_2}}.\]
The expression $\gamma_2^1\beta_2^1\alpha_2^{-1}\gamma_2^1\beta_1^{-1}$ is a string in $\Lambda$; its string module is as follows. 
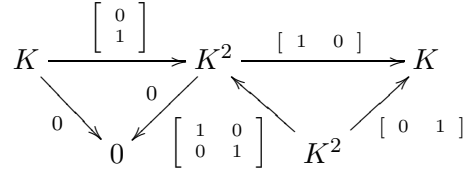
\[\xymatrix@C=20pt@R=20pt{
K\ar[rr]^{\tiny \left[\begin{array}{c}0\\ 1\end{array}\right]} \ar[dr]_{0} & &K^2 \ar[dl]_{0} \ar[rr]^{\tiny \left[\begin{array}{cc}1 & 0\end{array}\right]} & & K \\
&0 & & K^2 \ar[ul]^{\tiny \left[\begin{array}{cc}1 & 0\\ 0 & 1\end{array}\right]} \ar[ur]_{\tiny \left[\begin{array}{cc}0 & 1\end{array}\right]}}\]
Also, the strings $w_1:=\beta_1^{-1}\alpha_1^1\gamma_1^{-1}$, $w_2:=\beta^1_2\alpha_2^{-1}\gamma_2^1$, $w_2w_1$, and $w^2_2w_1$ are all bands of $\Lambda$, but the latter is not minimal. The band module $M(w_2, \lambda, 2)$ is as follows.  
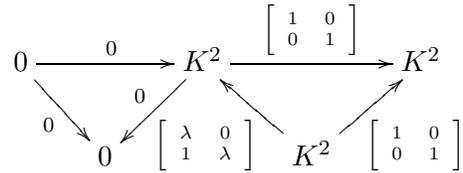
\[\xymatrix@C=20pt@R=20pt{
0\ar[rr]^{0} \ar[dr]_{0} & &K^2 \ar[dl]_{0} \ar[rr]^{\tiny \left[\begin{array}{cc}1 & 0\\ 0 & 1\end{array}\right]} & & K^2 \\
&0 & & K^2 \ar[ul]^{\tiny \left[\begin{array}{cc}\lambda & 0\\ 1 & \lambda\end{array}\right]} \ar[ur]_{\tiny \left[\begin{array}{cc}1 & 0 \\ 0 & 1\end{array}\right]}}\]
\end{ex}

We will frequently represent a string $w = \gamma_1^{\varepsilon_1} \cdots \gamma_d^{\varepsilon_d}$ or string module $M(w)$ as a diagram describing the action of $\Lambda$ on $M(w)$. We draw a southeast arrow for each symbol $\gamma_i^1$ in $w$, a southwest arrow for each symbol $\gamma_i^{-1}$ in $w$, and we arrange these arrows into a directed graph whose underlying graph is a path. For example, the string $\gamma^1_2\beta^1_2\alpha_2^{-1}\gamma^1_2\beta_1^{-1}$ from Example~\ref{example2} would be represented as follows.
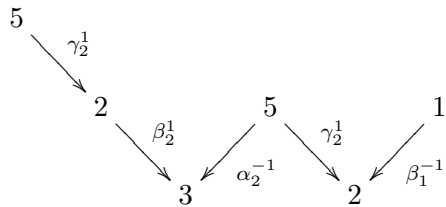
\[\xymatrix@C=20pt@R=20pt{ 
5 \ar[dr]^{\gamma^1_2} \\
& 2 \ar[dr]^{\beta^1_2} & & 5 \ar[dl]^{\alpha_2^{-1}} \ar[dr]^{\gamma^1_2} & & 1 \ar[dl]^{\beta_1^{-1}} \\
& & 3 & & 2}\]
We refer to this picture as the \textit{diagram} of $w$ or of $M(w)$.

\subsection{String algebras} 
It is straightforward to verify that string modules and band modules are indecomposable. As shown in \cite{wald1985tame}, string algebras $\Lambda$ are examples of algebras with the property that every indecomposable $\Lambda$-module is isomorphic to a string module or a band module. We are able to obtain our results on the maximal green sequences for such algebras by making use of the combinatorics of these string and band modules. We recall the definition of a string algebra now.  

We say that a finite dimensional algebra $\Lambda = KQ/I$ is a \textit{string algebra} if it satisfies the following properties:
\begin{itemize}
\item[(S1)] for any vertex $i \in Q_0$, there are at most two incoming and at most two outgoing arrows, and
\item[(S2)] for any arrow $\alpha \in Q_1$, there is at most one arrow $\beta$ and one arrow $\gamma$ such that $\alpha\beta \not \in I$ and $\gamma\alpha \not \in I$.
\end{itemize} 

We will be particularly interested in \textit{domestic} string algebras (i.e., string algebras with only finitely many bands); see for instance, \cite[Section 11]{ringel1995some}.  We will also work with a subclass of string algebras known as gentle algebras. A finite dimensional algebra $\Lambda$ is a \textit{gentle algebra} if, in addition to satisfying (S1) and (S2), it also satisfies the following two properties:
\begin{itemize}
\item[(G1)] for each arrow $\alpha \in Q_1$, there is at most one arrow $\beta$ and at most one arrow $\gamma$ such that $0 \neq \alpha\beta \in I$ and $0\neq \gamma\alpha \in I$,
\item[(G2)] the ideal $I$ may be generated by a finite set of paths of length two.
\end{itemize} 

\begin{ex}
The algebra $K(\xymatrix{1 \ar@<-1ex>[r]_\beta \ar@<1ex>[r]^\alpha & 2 })$ is a gentle algebra of domestic type because its only band is $\beta^{-1}\alpha^1$.
\end{ex}

\begin{ex}
The algebra $\Lambda$ from Example~\ref{example2} is a gentle algebra, and it is not of domestic type. Indeed, the set of strings $\{w_1(w_2w_1)^k \mid k \ge 1\}$ is an infinite set of bands of $\Lambda$.
\end{ex}


To understand complete forward hom-othogonal sequences of $\Lambda$, we will need to understand morphisms between string modules. The following is well-known; it gives a useful method for constructing morphisms between string modules.


\begin{prop}
Let $w$ and $w^\prime$ be two strings of a string algebra $\Lambda$. Then the dimension of $\text{Hom}_\Lambda(M(w), M(w^\prime))$ is the number of strings $u$ with multiplicities where $u$ is a substring of both $w$ and $w^\prime$ such that $M(u)$ is a quotient of $M(w)$ and $M(u)$ is a submodule of $M(w^\prime)$.
\end{prop}

Let $w =  \gamma_1^{\varepsilon_1}\cdots \gamma_d^{\varepsilon_d}$. Recall that $M(u)$ is a quotient of $M(w)$ if $u = \gamma_i^{\varepsilon_i}\cdots \gamma_j^{\varepsilon_j}$ is a substring of $w$ such that if $i\neq 1$ (resp., $j \neq d$), then $\varepsilon_{i -1} = -1$ (resp., $\varepsilon_{j+1} = 1$). Similarly,   $M(u)$ is a submodule of $M(w)$ if $u = \gamma_i^{\varepsilon_i}\cdots \gamma_j^{\varepsilon_j}$ is a substring of $w$ such that if $i\neq 1$ (resp., $j \neq d$), then $\varepsilon_{i -1} = 1$ (resp., $\varepsilon_{j+1} = -1$).

\section{Properties of bricks}\label{sec:bricks}




In this section we prove a number of lemmas needed to show the main results appearing in the next section.  We identify certain conditions on the modules that make them bricks.  In particular, we focus on string modules $M(\gamma)$ where $\gamma$ contains a substring $w$ for some band $w$.

The first lemma says that band modules cannot lie on a complete forward hom-orthogonal sequence. Therefore, we can omit band modules when studying maximal green sequences. Before proving the lemma we briefly review the connection between $\tau$-tilting modules, torsion classes, and bricks.  Recall that a $\Lambda$-module $N$ is {\it $\tau$-rigid} if $\text{Hom}_{\Lambda}(N, \tau N)=0$, where $\tau$ denotes the Auslander-Reiten translation in $\Lambda$-mod.   By \cite[Theorem 0.5]{AIR} there is a bijection between support $\tau$-tilting modules $N$ and functorially finite torsion classes  $\mathcal{T}(N)$ of $\Lambda$.  For a precise definition of support $\tau$-tilting modules see \cite[Definition 0.1]{AIR}, but for us it will only be important to note that such modules are $\tau$-rigid. Moreover, if two functorially finite torsion classes form a covering relation $\mathcal{T}(N_1)\lessdot \mathcal{T}({N}_2)$ in $\textsf{tors}\Lambda$  then $N_2$ differs from $N_1$ by an indecomposable summand $X\in \text{add}\,N_2 \setminus \text{add}\,N_1$  \cite[Theorem 2.30]{AIR}.  

Recall that Theorem~\ref{MGS_thm} gives a bijection between finite maximal chains in $\textsf{tors}\Lambda$ and certain sequences of bricks, and under this correspondence the module $M$ associated to the covering relation $\mathcal{T}(N_1)\lessdot \mathcal{T}({N}_2)$ is the unique smallest quotient of $X$ (with respect to dimension) that is a brick \cite[Proposition 4.9]{demonet2017lattice}.    In particular, $M$ is the unique smallest brick module in $\mathcal{T}({N}_2)\setminus \mathcal{T}({N}_1)$.  Also, note that a maximal green sequence in $\textsf{tors}\Lambda$ always comes from a chain of functorially finite torsion classes, see the proof of Theorem 5.3 in \cite{demonet2019survey} and references therein.


\begin{lem}\label{lemma-bands}
If $\Lambda$ is a string algebra, then no band module can lie on a maximal green sequence for $\Lambda$. 
\end{lem}

\begin{proof}
Let $M(w, \lambda, k)$ be a band module in $\Lambda$-mod.  If $k\geq 2$ then $M(w, \lambda, k)$ is not a brick, as there exists $f\in \text{End} \,M(w, \lambda, k)$ with image $M(w, \lambda, 1)$.  Therefore, it suffices to consider the case $k=1$.   We also note that band modules are not $\tau$-rigid as it is well-known that $\tau M(w, \lambda, k)=M(w, \lambda, k)$ \cite{butler1987auslander}.   This implies that band modules cannot appear as summands of a support $\tau$-tilting module.



Now suppose $M(w, \lambda, 1)$ lies on a maximal green sequence for some $\lambda\in K^*$.  Then there exist two functorially finite torsion classes $\mathcal{T}(N_1)\lessdot \mathcal{T}({N}_2)$ such that $X$ is a nonzero indecomposable module in $\text{add}\,N_2 \setminus \text{add}\,N_1$ and $M(w, \lambda,1)$ is the unique smallest quotient of $X$.  Let $f_{\lambda}: X \to M(w,\lambda, 1)$ denote this surjection, and note that because $X$ is a summand of a support $\tau$-tilting module $N_2$ it is $\tau$-rigid.  Therefore, $X$ is a string module.  In particular, $f_{\lambda}$ is a morphism of quiver representations where the linear maps in the quiver representation $X$ consist of multiplication by 0's and 1's while the linear maps in $M(w,\lambda, 1)$ consist of multiplication by 0's, 1's, and $\lambda$'s.   Therefore, $f_{\lambda}$ is a collection of linear maps with entries given by the parameter $\lambda$.   Let $f_{\lambda'}: X\to M(w, \lambda', 1)$ be the map obtained from $f_{\lambda}$ by replacing the parameter $\lambda$ with $\lambda'$.  It follows that $f_{\lambda'}$ is also a surjective morphism of quiver representations.  However, this contradicts uniqueness and minimality of the brick $M(w, \lambda, 1)$.
\end{proof}


By previous lemma it suffices to consider string modules when discussing maximal green sequences for string algebras.

Let $w$ be a band of length $n$.  Observe that every substring of $w^{N}$ may be expressed uniquely as $u_1^kv_1$ and $v_2u_2^k$ for some bands $u_1, u_2$ equivalent to $w$,  $k\geq 0$, and substrings $v_1,v_2$ of $u_1, u_2$ respectively of length less than $n$.
Given a band $w$, we say that a string $\gamma$ (or equivalently the string module $M(\gamma)$) is \textit{supported on $w^k$} if $u^k$  is a substring of $\gamma$ for some $u\sim w$ and $k\geq 1$.   


Let $\gamma = \gamma_1^{\varepsilon_1}\cdots\gamma_d^{\varepsilon_d}$ be a string and $\epsilon=\gamma_i^{\varepsilon_i}\cdots \gamma_j^{\varepsilon_j}$ be a substring of $\gamma$ for some $1 \le i \le j \le d$.  Then $\epsilon$ is called a {\it maximal $w$-substring} of $\gamma$ if $\epsilon$ is a substring of $w^N$ supported on $w$ and neither $\gamma_{i-1}^{\varepsilon_{i-1}}\epsilon$ nor $\epsilon\gamma_{j+1}^{\varepsilon_{j+1}}$ are substrings of $w^{N}$ for $i>1, j<d$ and $N$ large enough.  

{
\begin{ex}
In algebra $\Lambda$ from Example~\ref{example2}, the string $\gamma_2^1\beta_2^1\alpha_2^{-1}\gamma_2^1\beta_1^{-1}$
is supported on the band $w_2= \beta_2^1\alpha_2^{-1}\gamma_2^1$  but not $w_2^2$.  Moreover, $ \gamma_2^1\beta_2^1\alpha_2^{-1}\gamma_2^1$ is a maximal $w_2$-substring of $\gamma_2^1\beta_2^1\alpha_2^{-1}\gamma_2^1\beta_1^{-1}$.
\end{ex}
}

With this notation we have the following lemma.


\begin{lem}\label{lemma1}
Let $\Lambda$ be a string algebra with $M\in \Lambda$-mod.  Suppose $M=M(\gamma)$ is a brick supported on $w$, where $w$ is a band.  Let $\epsilon$ be a maximal $w$-substring of $\gamma$, then $M(\epsilon)$ is a submodule or a quotient of $M$. 
\end{lem}

\begin{proof}
By definition of $\epsilon$ being a maximal $w$-substring of $\gamma$, we may write $\epsilon=u^kv$ where $u\sim w$, $k\geq 1$, and $v$ is a proper substring of $u$.  Suppose $u=\alpha_1^{\varepsilon_1}\dots \alpha_n^{\varepsilon_n}$ and $v=\alpha_1^{\varepsilon_1} \dots \alpha_p^{\varepsilon_p}$ for some $p\in [0,n-1]$, where $p=0$ implies $v=e_j$ for some vertex $j$ is a string of length zero.  
Moreover, $\gamma=\gamma_1 \beta_1^{\varepsilon_1} \epsilon \beta_2^{\varepsilon_2} \gamma_2$ for some strings $\gamma_1, \gamma_2$ and arrows $\beta_1, \beta_2$.   Note that if $\epsilon$ appears at the beginning or the end of $\gamma$ then some of $\gamma_1\beta_1^{\varepsilon_1}$, $\beta_2^{\varepsilon_2}\gamma_2$ may be of length zero. 

If at least one of them has length zero, then the conclusion holds.  Thus, suppose that both $\gamma_1\beta_1^{\varepsilon_1}$, $\beta_2^{\varepsilon_2}\gamma_2$ are nontrivial. If $\beta_1^{\varepsilon_1}=\beta_1$ is an arrow then it suffices to show that $\beta_2^{\varepsilon_2}=\beta_2^{-1}$ is an inverse arrow, as this implies that $M(\epsilon)$ is a submodule of $M$.  The case when $\beta_1^{\varepsilon_1}=\beta_1^{-1}$ is an inverse arrow follows in a similar way, so we omit the detailed discussion. 

Suppose on the contrary that both $\beta_1^{\varepsilon_1}, \beta_2^{\varepsilon_2}$ are arrows.  
In particular, we have {that the diagram of $\gamma$ is of the following form}
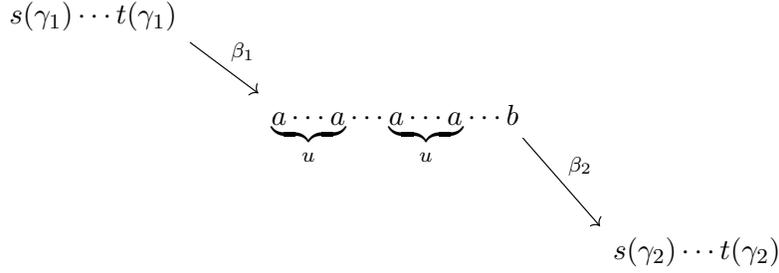
\[\begin{tikzcd}
s(\gamma_1) \cdots t(\gamma_1) \arrow[dr, start anchor = south east, end anchor = north west, "\beta_1"] \\
& \underbrace{a \cdots a}_u \cdots \underbrace{a \cdots a}_{u} \cdots b \arrow[dr, start anchor = {[yshift = 3ex, xshift = 1.5ex]}, end anchor = north west, "\beta_2"] \\
& & s(\gamma_2) \cdots t(\gamma_2)
  \end{tikzcd}\]
where $t(\beta_1)=a$ and $s(\beta_2)=b$ {and the sequence $a \cdots b$ represents the string $v$}.  Since $\epsilon$ is a maximal substring of $\gamma$ that is also a substring of $w^N$, we conclude that $\alpha_{p+1}^{\varepsilon_{p+1}}\not=\beta_2$ and $\alpha_n^{\varepsilon_n}\not=\beta_1$.  Moreover,  $\alpha_p^{\varepsilon_p}\beta_2$ and $\beta_1\alpha_1^{\varepsilon_1}$ are substring of $\gamma$ which implies $\alpha_p^{\varepsilon_p} \not= \beta_2^{-1}$ and    $\alpha_1^{\varepsilon_1}\not=\beta_1^{-1}$.



If $\alpha_{p+1}^{\varepsilon_{p+1}} =\alpha_{p+1}$ is an arrow, then $\alpha_{p+1}$ and $\beta_2$ are both arrows that start at $b$.  Because $\Lambda$ is a string algebra there are no other arrows starting in $b$.  
Therefore, $\alpha_p^{\varepsilon_p}=\alpha_p$ is an arrow ending in $b$.   Then $\alpha_p\alpha_{p+1}$ and $\alpha_p\beta_2$ are both substrings of $\gamma$, which is impossible because $\Lambda$ is a string algebra.  Hence, $\alpha_{p+1}^{\varepsilon_{p+1}}=\alpha_{p+1}^{-1}$ and $\beta_1 v  \alpha_{p+1}^{-1}$ is a substring of $\gamma$.  This shows that $M(v)$ is a submodule of $M$.  



Similarly, if $\alpha_n^{\varepsilon_n}=\alpha_n$  is an arrow, then $\alpha_n$ and $\beta_1$ are both arrows ending at $a$. Because $\Lambda$ is a string algebra there are no other arrows ending at $a$.   
Therefore, $\alpha_1^{\varepsilon_1}=\alpha_1$ is an arrow starting at $a$.  Then $\beta_1\alpha_{1}$ and $\alpha_n\alpha_1$ are both substrings of $\gamma$, which is impossible because $\Lambda$ is a string algebra. 
Therefore, $\alpha_n^{\varepsilon_n}=\alpha_n^{-1}$ is an inverse arrow and $ \alpha_n^{-1} v  \beta_2$ is a substring of $\gamma$.  This shows that $M(v)$ is also quotient of $M$, which yields a nonzero element of $\text{End} \, M$ with image $M(v)$.  This contradicting the assumption that $M$ is a brick and proves the lemma.
  \end{proof}

Given a string $\gamma = \alpha_1^{\varepsilon_1}\dots \alpha_d^{\varepsilon_d}$ for $d\geq 1$, we say that $\gamma$ is {\it directed} if $\varepsilon_i=\varepsilon_j$ for all $1\leq i,j\leq d$.  Otherwise, we say that $\gamma$ is {\it undirected}. A string module $M(\gamma)$ is {\it uniserial} if $\gamma$ is directed.

\begin{lem}\label{lemma2}
If $u$ is an undirected string such that $u^2$ {is a string}, then $u=w^k$ for some band $w$ and $k\geq 1$. 
\end{lem}

\begin{proof}
The string $u$ starts and ends at the same vertex, because $u^2$ {is a string}.  Therefore, to show that $u$ is a power of a (primitive) band $w$ it suffices to check that $u^m$ {is a string}  for all $m\geq 1$.

Let $u=\alpha_1^{\varepsilon_1}\dots \alpha_d^{\varepsilon_d}$ for $d\geq 1$.  Since $u$ is undirected, $d\geq 2$ and there exists $j$ for $1\leq j \leq d-1$ such that $\varepsilon_j \varepsilon_{j+1}=-1$.   Then $u'=\alpha_{j+1}^{\varepsilon_{j+1}} \dots \alpha_d^{\varepsilon_d}\alpha_1^{\varepsilon_1}\dots\alpha_{j}^{\varepsilon_{j}}$ {is a string}   because it is a substring of $u^2$.  Moreover, since $\varepsilon_j$ and $\varepsilon_{j+1}$ are {not equal}, it follows that the composition $(u')^m$ {is a string} for all $m\geq 1$.  This implies that $u^m$ is also {a string} for all $m\geq 0$, as it is a substring of $(u')^{m+1}$. 
\end{proof}

Let $w$ be a band.  The next lemma says that a module $M(\epsilon)$ coming from a maximal $w$-subword of a brick module is again a brick.  It is interesting to note that the proof is much more technical for the case when the subword contains only one copy of $w$ and requires additional assumptions.

\begin{lem}\label{lemmaBB}
Let $\Lambda$ be a string algebra with $M\in \Lambda$-mod.  Suppose $M=M(\gamma)$ is a brick supported on $w^k$ for a band $w$ and $k\geq 1$.  Let $\epsilon$ be a maximal $w$-substring of $\gamma$.  If one of the following conditions is satisfied 
\begin{itemize}
\item[(B1)] $k\geq 2$
\item[(B2)] $w$ is minimal, $M(\epsilon)$ is a submodule (resp., quotient) of $M$ and a quotient (resp., submodule) of some brick string module $M(\gamma')$
\end{itemize}
then $M(\epsilon)$ is a brick. 
\end{lem}

\begin{proof}
By Lemma~\ref{lemma1} the module $M(\epsilon)$ is either a submodule or a quotient of $M$.  Here we consider the case when $M(\epsilon)$ is a submodule of $M$, and the remaining case follows similarly. Suppose there exists $f\in \text{End}\, M(\epsilon)$ that is nonzero and not an isomorphism.  Moreover, we may assume that $\text{im}\,f$ is indecomposable.  Thus, $\text{im}\,f$ is both a submodule and a quotient of $M(\epsilon)$, and it is also a submodule of $M$.  Since $M$ is a brick, $\text{im}\,f$ cannot be a quotient of $M$.  This implies that $\text{im}\,f$ considered as a quotient of $M(\epsilon)$ is supported at a substring of $\epsilon$ appearing in the beginning or the end of $\epsilon$.  We suppose the former case holds.  In this situation, we write $\epsilon=u^kv$ where $u\sim w$, $k\geq 1$, and $v$ is a proper substring of $u$.  Then $\text{im}\,f = M(u^tv')$ for some $t\in [0,k]$ where $v'$ a proper substring of $u$ if $0 \le t \le k-1$ and $v'$ is proper substring of $v$ if $t = k$.  




Then $\gamma=\gamma_1\beta_1\epsilon \beta_2^{-1} \gamma_2$ for some strings $\gamma_1, \gamma_2$ and arrows $\beta_1, \beta_2$, where $\gamma_1\beta_1$ has nonzero length.   Let $u=\alpha_1^{\varepsilon_1}\dots \alpha_n^{\varepsilon_n}$, $v=\alpha_1^{\varepsilon_1} \dots \alpha_p^{\varepsilon_p}$, and $v'=\alpha_1^{\varepsilon_1}\dots\alpha_r^{\varepsilon_r}$ for some $p,r\in [0,n-1]$, where $p=0$ or $r=0$ implies $v$ or $v'$ has length zero.  Note, $\beta_1\not=\alpha_n$ because $\epsilon$ is a maximal $w$-substring of $M$, and $\varepsilon_{r+1} = 1$ because $\text{im}\,f$ is a quotient of $M(\epsilon)$.

Now we claim that $\varepsilon_n=-1$.  To show the claim, we suppose on the contrary that $\varepsilon_n=1$.  If we also have that $\varepsilon_1=1$, then $\beta_1\alpha_1$ is a substring of $\gamma$ and $\alpha_n\alpha_1$ is a substring of $w^2$.   Because $\beta_1\not=\alpha_n$, this contradicts the assumption that $\Lambda$ is a string algebra.  On the other hand, if $\varepsilon_n=1$ and $\varepsilon_1=-1$ then $\beta_1, \alpha_1, \alpha_n$ are three distinct arrows ending at the same vertex.   This also contradicts the assumption that $\Lambda$ is a string algebra.   Therefore, we conclude that $\varepsilon_n=-1$ as claimed. 
In particular, $\gamma$ contains a substring $\beta_1 \alpha_1^{\varepsilon_1}\dots \alpha_{n-1}^{\varepsilon_{n-1}}\alpha_n^{-1}$, so $M(\alpha_1^{\varepsilon_1}\dots \alpha_{n-1}^{\varepsilon_{n-1}})$ is a submodule of $M$.


If $t\geq 1$, then $\text{im}\,f$ has the structure of a quotient of $M(\epsilon)$ embedded in $M$ as shown in Figure~\ref{im_f_as_quotient}. Here $u$ starts and ends at vertex $a$ while $v'$ starts at $a$ and ends in some vertex $x$. 
Moreover, $\alpha_{r+1}, \alpha_n$ are arrows starting at $x$ and $a$ respectively, and $\beta_1$ is the arrow ending at $a$.  



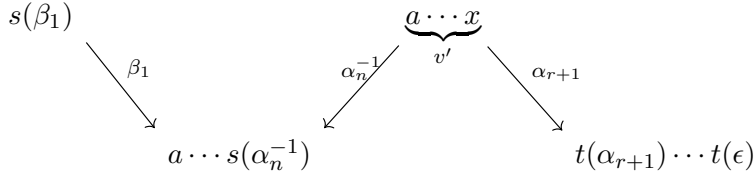
\begin{figure}[!htbp]
\[\begin{tikzcd}
s(\beta_1) \arrow[dr, start anchor = south east, end anchor = north west, "\beta_1"]  & & \underbrace{a \cdots x}_{v'}  \arrow[dl, start anchor = {[yshift = 1.5ex, xshift = .5ex]}, end anchor = north east, "\alpha_n^{-1}"{above}]  \arrow[dr, start anchor = {[yshift = 1.5ex, xshift = -.5ex]}, end anchor = north west, "\alpha_{r+1}"] \\
& a \cdots s(\alpha^{-1}_n) & & t(\alpha_{r+1}) \cdots t(\epsilon) 
  \end{tikzcd}\]
  \caption{The diagram of $M(\beta_1\epsilon)$. The sequence of vertices starting at the leftmost $a$ and ending at $x$ represents the string of $\text{im}\,f$. The orientation of $\alpha_{r+1}$ shows that $\text{im}\,f$ is a quotient of $M(\epsilon)$.}
  \label{im_f_as_quotient}
  \end{figure}


On the other hand, since $\text{im}\,f$ must also be a submodule of $M(\epsilon)$, its string must be a substring of $\epsilon$ as shown in Figure~\ref{im_f_as_submodule}. Since $\text{im}\,f = M(u^tv')$ is a proper submodule of $M(\epsilon)$ and $u^tv'$ is a substring of $\epsilon$ appearing in the beginning of $\epsilon$, we have that $\delta_1 = \beta_1$ and $\delta_2^{-1}$ appears in $\epsilon$. However, this implies that $\delta_2^{-1} = \alpha_{r+1}$, which is not possible. This yields a contradiction in the case where $t \ge 1$.

\begin{figure}[!htbp]
\[\begin{tikzcd}
s(\delta_1) \arrow[dr, start anchor = south east, end anchor = north west, "\delta_1"]  & & t(\delta_2^{-1}) \cdots t(\epsilon)  \arrow[dl, start anchor = {[yshift = 1ex, xshift = -2.5ex]}, end anchor = north east, "\delta_2^{-1}"{above}]   \\
& a \cdots x 
  \end{tikzcd}\]
  \caption{The diagram of $M(\delta_1\epsilon)$. The sequence of vertices starting at $a$ and ending at $x$ represents the string of $\text{im}\,f$. The orientation of $\delta_{2}^{-1}$ shows that $\text{im}\,f$ is a submodule of $M(\epsilon)$.}
  \label{im_f_as_submodule}
  \end{figure}
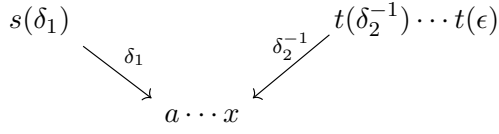

Now, we consider the case $t=0$, so $\text{im}\,f=M(v')$.  Recall that $M(v')$ is a submodule of $M$.  If $k\geq 2$ or $r< p$ then $\gamma$ contains a substring $\alpha_{n}^{-1}  v'  \alpha_{r+1}$.  Thus, $M(v')$ is also a quotient of $M$.   This contradicts the assumption that $M$ is a brick, and proves the lemma if condition (B1) is satisfied. 



Now suppose condition (B2) is satisfied. By above it remains to consider the case $t=0, k=1$, and $r\geq p$.  Then $\text{im}\,f = M(v')$ and $M(\epsilon)$ appear inside $M$ as shown in Figure~\ref{t_0_k_1}.

\begin{figure}[!htbp]
\[\begin{tikzcd}
s(\beta_1) \arrow[dr, start anchor = south east, end anchor = north west, "\beta_1"]  & & & &  t(\beta_2^{-1}) \arrow[dl, start anchor = {[yshift = 1ex, xshift = -.5ex]}, end anchor = north east, "\beta_2^{-1}"{above}]   \\
& \underbrace{\underbrace{a \cdots b}_{v} \cdots x}_{v'} \arrow[dr, start anchor = {[yshift = 5ex, xshift = -.5ex]}, end anchor = north west, "\alpha_{r+1}"] & & \underbrace{a \cdots b}_{v} \arrow[dl, start anchor = {[yshift = 3ex, xshift = 0ex]}, end anchor = north east, "\alpha_n^{-1}"{above}] \\ & & \cdots  
  \end{tikzcd}\]
  \caption{Here, we show the diagram of $\beta_1\epsilon\beta_2^{-1}$.The string $\epsilon$ starts at the leftmost vertex $a$ and ends at the rightmost vertex $b$.}
  \label{t_0_k_1}
  \end{figure}
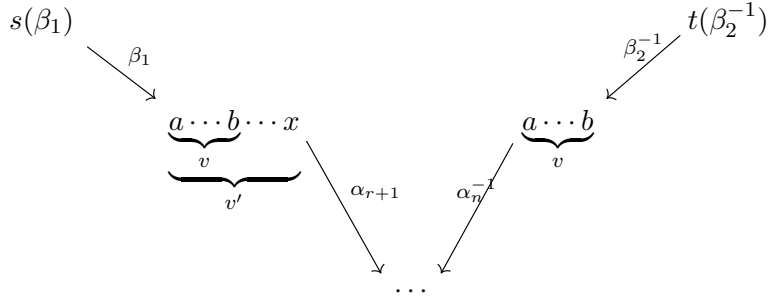


We claim that $\varepsilon_{p+1}=1$. {Otherwise, if $\varepsilon_{p+1} = -1$, then $\beta_2 = \alpha_{p+1}$ because $\Lambda$ is a string algebra. Then since $\epsilon$ is a maximal $w$-substring of $\gamma$ this gives a contradiction unless $\beta_2^{-1}\gamma_2$ has length zero.}   Thus, $\gamma$ also contains a substring $\alpha_n^{-1}v$ at the very end.  This shows that $M(v)$ is both a quotient and a submodule of $M$, contrary to $M$ being a brick. This shows the claim. 


By condition (B2) there exists a brick string module $M(\gamma')$ such that $M(\epsilon)$ is a quotient of $M(\gamma')$. Then $M(v')$ is also a quotient of $M(\gamma')$, see Figure~\ref{t_0_k_1}.  If $\text{im}\,f=M(v')$ as a submodule of $M(\epsilon)$ does not embed at the very end of $M(\epsilon)$, then $M(v')$ is also a submodule of $M(\gamma')$.  Therefore, we obtain that $M(v')$ is both a submodule and a quotient of $M(\gamma')$, contrary to $M(\gamma')$ being a brick.  This shows that the the string $\epsilon$ ends in $v'$ or $(v')^{-1}$.  


First, suppose $\epsilon$ ends in $(v')^{-1}$.  Since the string $v$ appears at the beginning of $v'$, we conclude that $\epsilon$ ends in $v^{-1}$.   However, we also have that $\epsilon=uv$ ends in the string $v$.  This shows that $v=v^{-1}$, which is only possible if $v=e_a$ for a vertex $a$.  Then we have $\epsilon=u$ which starts in $v'$ and ends in $(v')^{-1}$.  Therefore, $u=v'z_1=z_2(v')^{-1}$ for some strings $z_1, z_2$.  Since $u$ is a band, we obtain $u^2 = z_2 (v')^{-1}v' z_1$ is a string.  This implies that $v'=e_a$. 
We also have that $M(e_a)$ as a submodule of $M(\epsilon)$ embeds at the end of $\epsilon=u$.  But $u$ ends with an inverse arrow $\alpha_{n}^{-1}$, which implies that $M(e_a)$ is a quotient of $M(\epsilon)$ and not a submodule.  This is a contradiction, which shows that $\epsilon$ does not end in $(v')^{-1}$. 

Now, suppose $\epsilon$ ends in $v'$.  We observe that $v'\not=v$ because $\epsilon$ ends in $\alpha_{n}^{-1}v$  so $M(v)$ is a quotient of $M(\epsilon)$.   However, we also have that $\epsilon$ ends in $v'$ such that $M(v')$ is a submodule of $M(\epsilon)$, so we cannot have $v'=v$.  Therefore, $\ell(v')>\ell (v)$. Moreover, $\ell(v')<\ell(u)$ since $t=0$.  Since $\epsilon$ starts and ends in a string $v$, we conclude that $v'$ also starts and ends in a string $v$.  We consider two cases based on the relative lengths of $v$ and $v'$.

If $\ell (v')\geq 2\ell(v)$  then we can write $v'=vzv$ for some string $z$.  Note that $vz$ is an undirected string because $\varepsilon_{p+1}=1, \varepsilon_n = -1$, which means that $z$ starts with an arrow $\alpha_{p+1}$ and ends with an inverse arrow $\alpha_n^{-1}$.   Moreover, $\ell (vz)\geq 1$ because $\ell(v')>\ell(v) \geq 0$.  We can write $u=v'z_0$ for some string $z_0$ with $\ell(z_0)\geq 1$.   Then $\epsilon=uv = vzvz_0v$.  Since $\epsilon$ ends in $v'$ we an also write $\epsilon=vz_1v'$ for some string $z_1$ with $\ell(z_0)=\ell (z_1)$.   Combining the two expressions for $\epsilon$ we conclude that $zvz_0=z_1vz$ and we can write 
$$u^2=\epsilon zvz_0 = vz_1vzvzvz_0=vz_1(vz)^2vz_0.$$

If $\ell (z)\leq \ell(z_0)$  then $\ell((vz)^2)\leq \ell (u)$.  Thus $vz$ is an undirected string such that $(vz)^2$ is {a string}.  By Lemma~\ref{lemma2}, we have that $vz$ is a proper power of a band,  so $u$  (and hence $w$) is not a minimal band. This contradicts the assumption in (B2).   If $\ell(z) > \ell (z_0)=\ell (z_1)\geq 1$ then the equation $zvz_0=z_1vz$ implies that $z=z_1z_2$ for some string $z_2$ with $\ell(z_2)\geq 1$.    Note also that $\ell(vz_1)\geq 1$ and we can write 
\[u=vz_1vz=vz_1vz_1z_2=(vz_1)^2z_2.\]

If $vz_1$ is undirected, then we obtain a contradiction to $w$ being minimal as before.  Suppose on the contrary that $vz_1$ is directed. Then $(vz_1)^2$ is also directed, and because $\varepsilon_{p+1}=1$ we conclude that $(vz_1)^2$ is a composition of arrows instead of inverse arrows.  Since $u$ is undirected, then we may write $z_2=z_3z_4$ for some strings $z_3, z_4$ such that $(vz_1)^2 z_3$ is a maximal directed substring of $u$ appearing in the beginning of $u$.  Then $z_4$ starts with an inverse arrow $\rho_1^{-1}$, and let the last arrow in $z_1$ be $\rho_2$. Then $\rho_2 vz_1z_3 \rho_1^{-1}$ is a substring of $\epsilon$ and also $\gamma'$ such that $M(vz_1z_3)$ is a submodule of $M(\gamma')$. Moreover, the string $\rho_2 vz_1z_3 \rho_1^{-1}$ is not an initial substring of $\epsilon$.

We now claim that $z_3$ is an initial substring of $(vz_1)^m$ for some $m\geq 1$. Note that both $vz_1$ and $z_3$ are directed strings and $vz_1$ has nonzero length. Then we can write $vz_1 = \omega_1 \cdots \omega_r$ and $z_3 = \eta_1 \cdots \eta_s$ where each $\omega_i$ and each $\eta_j$ is an arrow, with the caveat that $z_3$ may be of length zero.
Since $(vz_1)^2$ and $vz_1z_3$ are strings, we know that $\omega_r\omega_1$ and $\omega_r\eta_1$ are strings. Because $\Lambda$ is a string algebra, we have that $\omega_r\omega_1 = \omega_r\eta_1.$ We conclude that $\omega_1 = \eta_1$. Similarly, we know that $\omega_1\omega_2$ and $\eta_1\eta_2$ are strings and that $\eta_1\eta_2 = \omega_1\eta_2$. We obtain that $\omega_1\eta_2 = \omega_1\omega_2$ and so $\omega_2 = \eta_2$. Continuing this process, we see that $z_3$ is an initial substring of $(vz_1)^m$ for some $m \ge 1$.

Since $M(\epsilon)$ is a quotient of $M(\gamma')$ we note that $\beta_3^{-1}\epsilon$ is a substring of $\gamma'$ where $\beta_3$ is either an arrow or a string of length zero, and the latter situation occurs only if $\epsilon$ and $\gamma'$ have the same initial vertex. 
Then we also see that $\beta_3^{-1}vz_1z_3 \rho_3$ is an initial substring of $\gamma'$ for some arrow $\rho_3$, since $(vz_1)^2z_3$ is a directed string that is a composition of arrows and since $z_3$ is substring of $(vz_1)^m$ for some $m \ge 1$. We conclude that $M(vz_1z_3)$ is also quotient of $M(\gamma')$, which is a contradiction. This shows the lemma in the case $\ell (v')\geq 2\ell(v)$.  

Lastly, it remains to consider the case $\ell (v')< 2\ell(v)$.  Then $v'=vz_1=z_2v$ for some strings $z_1, z_2$ such that $1\leq \ell(z_1)=\ell(z_2)<\ell (v)$.   In this case, we can write $v=z_2z_3$ for some string $z_3$ with $\ell(z_3)\geq 1$.  Then 
$$v'=z_2v=z_2z_2z_3=(z_2)^2z_3.$$
If $z_2$ is undirected we again obtain a contradiction to $u$ being minimal as before.  If $z_2$ is a directed string that is a composition of arrows, then we get a contradiction to $M(\gamma')$ being a brick as in the previous case.  Otherwise, if $z_2$ is a directed string that is a composition of inverse arrows, then we get a similar contradiction to $M=M(\gamma)$ being a brick.  This completes the proof in the second case, and shows the lemma. 
\end{proof}

The next few lemmas relate various properties of  brick string modules supported on $w$ and the corresponding band modules $M(w, \lambda, N)$ for $\lambda\in K^*$.


\begin{lem}\label{lemma_sub-quo}
Let $M(\epsilon)$ be a string module over a string algebra $\Lambda$ supported on $w$ where $w$ is a band.  Suppose $M(\epsilon)$ is a quotient (resp., submodule) of $M(w,\lambda,N)$ for some $N$ and a submodule (resp., quotient) of $M(\gamma)$ for some string $\gamma$. Then $\epsilon$ is a maximal $w$-substring of $\gamma$.
\end{lem}

\begin{proof}
We consider the case when $M(\epsilon)$ is a quotient of $M(w,\lambda,N)$ for some $N$.  Then $\beta_1^{-1}\epsilon\beta_2$ is a substring of $w^{N+1}$ where $\beta_1, \beta_2$ are arrows.  
On the other hand, since $M(\epsilon)$ is a submodule of $M(\gamma)$ it follows that $\alpha_1\epsilon\alpha_2^{-1}$ is a substring of $\gamma$, where $\alpha_1, \alpha_2$ are arrows or constant paths of length zero.  In particular, we see that $\beta_1^{-1} \not = \alpha_1$ and $\beta_2\not=\alpha_2^{-1}$. Therefore, $\epsilon$ cannot be extended in $M(\gamma)$ to a larger substring that would also be a substring of $w^{N+1}$.  This shows the claim.
\end{proof}

\begin{lem}\label{lemma4}
Let $w$ be a band, and let $\epsilon$ be a substring of $w^{N}$ supported on $w$ for some $N \ge 1$.  If $M(\epsilon)$ is a brick, then it is a submodule or a quotient of $M(w,\lambda, N+1)$ for any $\lambda \in K^*$. 
\end{lem}

\begin{proof}
Since $\epsilon$ is a substring of $w^{N}$ supported on $w$ then we can write $\epsilon=u^kv$ where $k\geq 1$, $u\sim w$, and $v$ is a proper substring of $u$. 
Let $u=v\alpha^{\varepsilon_{\alpha}} v' \beta^{\varepsilon_{\beta}} $ where $\alpha,\beta$ are arrows while $v'$ is a string.  Suppose $\varepsilon_{\alpha}=1$, then $\alpha$ starts at $t(v)$ and ends in $s(v')$.  In particular, $M(v)$ is a quotient of $M(\epsilon)$ since $u$ appears in the beginning of $\epsilon$. Because $v$ also appears at the end of $\epsilon$, if $M(\epsilon)$ is not a quotient of $M(w,\lambda,N+1)$ then $\varepsilon_\beta=1$ and the arrow $\beta$ starts in $t(v')$ and ends in $s(v)$.  In this case, $M(v)$ is a submodule of $M(\epsilon)$ as  $\beta v$ appears at the very end of $\epsilon$.  This contradicts the assumption that $M(\epsilon)$ is a brick.  Thus, $M(\epsilon)$ is a quotient of $M(w,\lambda,N+1)$ as desired.  
The case when $\varepsilon_{\alpha}=-1$ follows similarly.  
\end{proof}



\begin{lem}\label{lemC2}
Let $\Lambda$ be a string algebra with a minimal band $w$ and a string $\epsilon=u^kv$ where $k\geq 1$, $u\sim w$, and $v$ is a proper substring of $u$.  If $u=u_0^2u'$ for some string $u_0$ with $\ell(u_0)\geq 1$, then $M(\epsilon)$ is not a brick.  
\end{lem}

\begin{proof}
If $u_0$ is undirected then by Lemma~\ref{lemma2} we conclude that $u_0$ is equals a power of some band, contrary to $w$ being minimal.  Thus, suppose $u_0$ is a directed string, and let $u_0^2v'$ be a maximal directed substring in $u$ appearing at the start of $u$.  Note that $u$ is not directed since it is a band and $\Lambda$ is finite dimensional. This implies that $u_0^2v'$ is a proper initial substring of $u$.  

Using an argument similar to what appears in the proof of Lemma 3.4 and the fact that $u_0^2v'$ is a directed string, one shows that $v'$ is an initial substring of $u_0^{m}$ for some $m \ge 1$. Therefore, if $u_0 = \omega_1 \cdots \omega_s$, we can write $u_0^2v' = u_0^2u_0^m\omega_1 \cdots \omega_r$ for some $r \le s$ and some $m \ge 0$. We now see that $u_0v' = u_0^{m+1}\omega_1\cdots \omega_r$ is an initial substring of $u_0^2v'$. This initial substring of $u_0^2v'$ identifies $M(u_0v')$ as a quotient of $M(\epsilon)$ since $u_0^2v'$ is a directed string. 

We can also view $u_0v' = u_0^{m+1}\omega_1\cdots \omega_r$ as a terminal substring $u_0^2v'$. By the maximality of $u_0^2v'$, this terminal substring of identifies $M(u_0v')$ as a quotient of $M(\epsilon$). This shows that $M(\epsilon)$ is not a brick. 
\end{proof}

Next is the last result of this section.  It says that we can enlarge certain brick modules supported on $w$ by adding another copy of a band, and the resulting module remains a brick.  Similarly to the proof of Lemma~\ref{lemmaBB}, the arguments become more technical in the case $k=1$. 

\begin{lem}\label{lemmaCC}
Let $\Lambda$ be a string algebra.  Suppose $w$ is a minimal band in $\Lambda$ and $\epsilon=u^kv$ where $k\geq 1$, $u\sim w$, and $v$ is a proper substring of $u$.  If one of the following conditions is satisfied 
\begin{itemize}
\item[(C1)] $k\geq 2$ and $M(\epsilon)$ is a brick
\item[(C2)] $k\geq 1$, $M(\epsilon)$ is a brick and a quotient or a submodule of a brick $M(z)$ where $z$ is a substring of $w^{N}$, for some $N$, supported on $u^{k+1}$
\end{itemize}
then $M(u\epsilon)$ is a brick. 
\end{lem}

\begin{proof}
Suppose $M(\epsilon)$ is a brick.  Let $u=u'\alpha^{\varepsilon_{\alpha}}$ where $u'$ is a string and $\alpha$ is an arrow.  Observe that $M(\epsilon)$ is either a quotient or a submodule of $M(u\epsilon)$ depending on the sign of $\varepsilon_{\alpha}$.  We consider the former case when $\varepsilon_{\alpha}=-1$, so $M(\epsilon)$ is a quotient of $M(u\epsilon)$.  The proof of the latter case follows similarly.  If $\varepsilon_{\alpha}=-1$,  we have the following short exact sequence 
$$0\to M(u')\xrightarrow{i} M(u\epsilon) \xrightarrow{\pi} M(\epsilon)\rightarrow 0.$$
Since $\epsilon$ starts with $u=u'\alpha^{-1}$ then $M(u')$ is also a submodule of $M(\epsilon)$.  Let $i': M(u')\to M(\epsilon)$ denote this inclusion. 

We want to show that $M(u\epsilon)$ is a brick, so assume there exists some $f\in \text{End}\, M(u\epsilon)$ nonzero and not an isomorphism.  Moreover, we may suppose that $\text{im}\, f$ is indecomposable.   Let $j: \text{im}\,f \to M(u\epsilon)$ denote the corresponding inclusion.  

If $\text{im}\,f$ is also a quotient of $M(\epsilon)$, then there exists a surjective map $\rho: M(\epsilon)\to \text{im}\,f$.   Consider the following composition 
$$M(\epsilon)\xrightarrow{\rho} \text{im}\,f\xrightarrow{j} M(u\epsilon) \xrightarrow{\pi} M(\epsilon),$$
which yields an element of $\text{End} \, M(\epsilon)$.   This map cannot be an isomorphism, hence it must be zero.   Because $\pi j \rho = 0$ and $\rho$ is surjective, we conclude that $\pi j =0$.  This implies that $j$ factors through $\text{ker}\,\pi = M(u')$.   In particular, there exists a map $g: \text{im}\,f \to M(u')$ such that $ig=j$.   We obtain another composition of morphisms as follows: 
$$M(\epsilon)\xrightarrow{\rho} \text{im}\,f \xrightarrow{g} M(u')\xrightarrow{i'} M(\epsilon).$$
Again this cannot be an isomorphism, so it must be zero.  Since $\rho$ is surjective and $i'$ is injective, we conclude $g=0$ and hence $j=0$.  This implies that $\text{im}\,f=0$ so $f$ is the zero map. This contradicts that $\text{im}\,f$ is nonzero. So $\text{im}\,f$ is not a quotient of $M(\epsilon)$.


Since $\text{im}\,f$ is not a quotient of $M(\epsilon)$ but it is a quotient of $M(u\epsilon)$, we conclude that $\text{im}\,f=M(u''u^{k-1}vv')$ for some $v',u''$ substrings of $u$ of length greater than zero.  Note that $u$ ends in $u''$ and $u^2$ starts with $vv'$.   Suppose condition (C1) is satisfied, so $k\geq 2$.  Then we see that $M(u''u^{k-2}vv')$ is a quotient of $M(\epsilon)$.  Similarly, because $\text{im}\,f$ is a submodule of $M(u\epsilon)$ then $M(u''u^{k-2}vv')$ is also a submodule of $M(\epsilon)$.  This contradicts the assumption that $M(\epsilon)$ is a brick and proves the first part of the lemma. 


Now suppose condition (C2) is satisfied.  By above it suffices to consider the case $k=1$, and we have $\text{im}\,f =  M(u''vv')$ where $v',u''$ have length greater then zero, $u$ ends in $u''$ and $u^2$ starts with $vv'$.   Suppose $u$ starts and ends at vertex $a$, the string $v$ ends in vertex $b$, while $v'$ ends in $y$ and $u''$ starts in some vertex $x$.  We give the corresponding diagrams in Figures \ref{fig:4} and \ref{fig:5}. 

\begin{figure}[!htbp]
\[\begin{tikzcd}
a\cdots b \arrow[dr, start anchor = south east, end anchor = north west, "\beta"]  & & a \cdots b  \arrow[dl, start anchor = {[yshift = 0ex, xshift = -1.5ex]}, end anchor = north east, "\alpha_{\ }^{-1}"{above}]  \arrow[dr, start anchor = {[yshift = 0ex, xshift = 1.5ex]}, end anchor = north west, "\beta"] & & \underbrace{a \cdots b}_{v}  \arrow[dl, start anchor = {[yshift = 1.5ex, xshift = .5ex]}, end anchor = north east, "\alpha_{\ }^{-1}"{above}]  \\
& t(\beta) \cdots s(\alpha^{-1}) & & t(\beta) \cdots t(\alpha^{-1}) 
  \end{tikzcd}\]
  \caption{The diagram of $M(u\epsilon) = M(u^2v)$.}
  \label{fig:4}
  \end{figure}
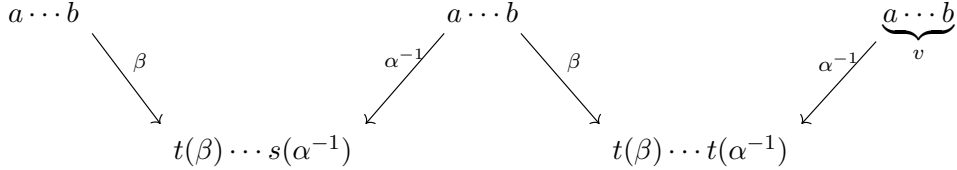
  
  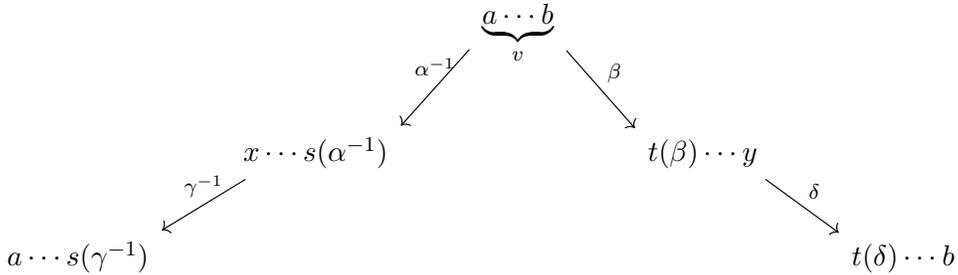
\begin{figure}[!htbp]
\[\begin{tikzcd}
 & & \underbrace{a \cdots b}_v  \arrow[dl, start anchor = {[yshift = 1ex, xshift = 0ex]}, end anchor = north east, "\alpha_{\ }^{-1}"{above}]  \arrow[dr, start anchor = {[yshift = 1ex, xshift = 0ex]}, end anchor = north west, "\beta"] & &  \\
& x \cdots s(\alpha^{-1})  \arrow[dl, start anchor = {[yshift = 0ex, xshift = -1.5ex]}, end anchor = north east, "\gamma_{\ }^{-1}"{above}] & & t(\beta) \cdots y  \arrow[dr, start anchor = {[yshift = 0ex, xshift = 1.5ex]}, end anchor = north west, "\delta"]  \\
a\cdots s(\gamma^{-1}) & & & & t(\delta) \cdots b
  \end{tikzcd}\]
  \caption{The diagram of $\text{im}\,f$ as a quotient of $M(u\epsilon)$.}
  \label{fig:5}
  \end{figure}



Since $\varepsilon_{\alpha}=-1$ and $M(\epsilon)$ is a brick, Lemma~\ref{lemma4} implies that $M(\epsilon)$ is a quotient of $M(w,\lambda, 2)$ for any $\lambda$ and, thus, also a quotient of $M(w,\lambda,N)$ for any $N \ge 2$.   By condition (C2) we then conclude that $M(\epsilon)$ must also be a quotient (and not a submodule) of some brick module $M(z)$ supported on $u^2$. 
Since $M(z)$ is supported on $u^2$, let $z'$ be obtained from $z$ by replacing some factor of $u^2$ appearing in $z$ with $u^3$. The fact that $z$ is supported on $u^2$ allows us to write $z = u_0^2v_0$ and $z' = u_0^3v_0$ where $u_0 \sim u$. 

Next, $M(z)$ is a brick and so (C1) is satisfied. Applying the theorem in this case to $M(z)$, we conclude that $M(z')$ is also a brick.  Note that $M(u\epsilon)$ is then a quotient of $M(z')$.  Thus $\text{im}\,f$ is also a quotient of $M(z')$, but not a submodule.  This implies that as a submodule $\text{im}\,f$ embeds in $M(u\epsilon)$ starting at an endpoint of $u\epsilon$ that is not an endpoint of $M(z')$.  We may suppose that $\text{im}\,f$ as a submodule embeds starting at the left endpoint of $u\epsilon$.   Otherwise, we may rewrite $\epsilon$ as $v_0u_0$ where $u_0\sim u$, and then work with $u\epsilon=v_0u_0^2$.

Recall that $\text{im}\,f =  M(u''vv')$ where $v',u''$ have length greater then zero, and by above $\text{im}\,f$ as a submodule embeds starting at the left endpoint of $u\epsilon$.   In particular, $u\epsilon$ starts with $u''vv'$ or its inverse $(u''vv')^{-1}$.  

If $u''vv'$ has length strictly greater than $n$, where $\ell(u)=n$, then we can similarly rewrite $u''vv'=uv_0$ or $(u''vv')^{-1}=uv_0$ where $\ell(v_0)\geq 1$.   Then $\epsilon$ starts with $v_0$, and, moreover, $M(v_0)$ is a submodule of $M(\epsilon)$ since  $\text{im}\,f$ is a submodule of $M(u\epsilon)$. Also, since $\text{im}\,f=M(uv_0)$ is a quotient of $M(u\epsilon)$ and $\varepsilon_{\alpha}=-1$ we conclude that $v_0$ is also a quotient of $M(\epsilon)$.   This yields a contradiction.

Therefore, the length of $u''vv'$ is at most $n$, where $u$ ends in $u''$ and starts with $vv'$. In this case, we can write $u=vv'z_1u''$ for some string $z_1$ with $\ell(z_1) \ge 0$.   
Because $\text{im}\,f$, as a submodule, embeds at the left endpoint of $M(u\epsilon)$ and also $M(\epsilon)$, we obtain $u=u''vv'z_2$ or $u=(u''vv')^{-1}z_2$ for some string $z_2$ with $\ell(z_2)\geq 0$.  





First suppose $u=(u''vv')^{-1}z_2=vv'z_1u''$.  This implies $(vv')^{-1}=vv'$, and so $v=v'=e_a$ has length zero.  This contradicts the earlier statement that $\ell(v')\geq 1$.

Now it remains to consider the case $u=u''vv'z_2 = vv'z_1u''$.  In what follows, we will rewrite $u$ so that we obtain a contradiction either to $w$ being minimal, when $u$ or a band equivalent to $u$ contains a square of an undirected string, see Lemma~\ref{lemma2}, or we obtain a contradiction to $M(\epsilon)$ being a brick, when $u$ starts with a square of a string of nonzero length, see Lemma~\ref{lemC2}.    


We have $u=u''vv'z_2 = vv'z_1u''$, and first suppose that  $\ell(v)\geq \ell(u'')$.  Then we can write $v=u''u_0$ for some string $u_0$, and observe that $u$ starts in $(u'')^2$, contradicting Lemma~\ref{lemC2}.  Note that $\ell(u'')\geq 1$ because of an earlier statement.  Otherwise if $\ell(v)< \ell(u'')$ we can write $u''=vv_0$ for some string $v_0$ with $\ell(v_0) \ge 1$ and obtain $u=vv_0vv'z_2=vv'z_1u''$. This means that either $v_0$ starts with $v'$ or $v'$ starts with $v_0$. If $v'$ starts with $v_0$ then $u$ starts with $(vv_0)^2=(u'')^2$, and we obtain a contradiction to Lemma~\ref{lemC2}. Otherwise, if $v_0$ starts with $v'$, then $v_0=v'v_2$ for some string $v_2$ with $\ell(v_2) \ge 0$. 
Hence, 
$z_1u''=z_1vv'v_2$ and
$$u=vv'v_2vv'z_2=vv'z_1u''=vv'z_1vv'v_2.$$  
If $\ell(v_2)\leq \ell(z_1)$, then $z_1=v_2z_3$ for some string $z_3$. Then 
$$u\sim v_2 vv'z_1vv'\sim (vv'v_2)^2z_3$$
This gives a contradiction to $w$ being minimal provided that $vv'v_2$ is not directed, see Lemma~\ref{lemma2}.  


Next, we prove that $vv'v_2$  is not directed by showing that this string contains an arrow and an inverse arrow.   First, observe that $\ell(vv'v_2)\geq 1$ since $\ell(v')\geq 1$.  Since $u$ ends in $vv'v_2$ and also $u$ ends in an inverse arrow $\alpha^{-1}$ then we conclude that $vv'v_2$ contains an inverse arrow $\alpha^{-1}$.    Because $M(\epsilon)$ is a quotient of $M(w,\lambda,N)$ for $N\geq 2$ then so is $M(v)$.   Then $v'$ starts in an arrow, because $u$ starts with $vv'$ and $M(v)$ is a quotient of $M(\epsilon)=M(uv)$.  This means that $vv'v_2$ contains an arrow.  This shows the claim  that $vv'v_2$ is not directed. 

Finally, if $\ell(v_2)>\ell(z_1)$ we can write $v_2=z_1v_3$. Then  $u$ starts with $(vv'z_1)^2$, contradicting Lemma~\ref{lemC2}. 
This completes the proof in the case (C2).
\end{proof}

\section{Maximal green sequences}\label{sec:mgs}

In this section we present the main results of the paper. First, we prove that no module supported on a square of a band may appear on a maximal green sequence.  As a consequence, we conclude that domestic string algebras admit at most finitely many maximal green sequences.  Then we identify certain string algebras that admit at least one maximal green sequence.  Here, we find a particular ordering on the simple $\Lambda$-modules that ensures it can be completed to a maximal green sequence.

We begin with the following lemma prohibiting certain string modules from appearing on a maximal green sequence. It is a special case of the more general result appearing in the next theorem.

\begin{lem}\label{lemmaw2}
Let $\Lambda$ be a string algebra. Suppose $w$ is a minimal band and $\epsilon=u^kv$ where $k\geq 2$, $u\sim w$, and $v$ is a proper substring of $u$.  Then the module $M(\epsilon)$ cannot lie on a maximal green sequence for $\Lambda$. 
\end{lem}

\begin{proof}

Suppose $\mathcal{M}$ is a maximal green sequence for $\Lambda$ in which a module $M(\epsilon)$ that satisfies the conditions in the statement of the lemma appears. Such a module $M(\epsilon)$ is necessarily a brick, and Lemma~\ref{lemma4} implies that $M(\epsilon)$ is a quotient or a submodule of $M(w, \lambda, k+1)$ for any $\lambda$. Suppose $M(\epsilon)$ is a quotient, and the other case follows dually. By possibly replacing $M(\epsilon)$ with another module satisfying the conditions in the statement of the lemma, we can assume that $M(\epsilon)$ is rightmost such module in $\mathcal{M}$. The complete forward hom-orthogonal sequence $\mathcal{M}$ may thus be written as follows
$$\mathcal{M}:\,\,\,  M_{-q}, M_{-q+1}, \dots, M_{-1}, M_0=M(\epsilon), M_1, \dots, M_{p-1}, M_p$$
where $p,q\geq 0$.  Condition (C1) of Lemma~\ref{lemmaCC} is satisfied by $M(\epsilon)$, so $M(u\epsilon)$ is also a brick.  Moreover, since $M(\epsilon)$ is a quotient of $M(w, \lambda, k+1)$, the module $M(\epsilon)$ is also a quotient of $M(u\epsilon)$ and we have the following short exact sequence in $\Lambda$-mod: 
$$0\to \text{ker}\,\pi \xrightarrow{\iota} M(u\epsilon) \xrightarrow{\pi} M(\epsilon)\to 0.$$
Note that $\text{ker}\,\pi$ is also a submodule of $M(\epsilon)$, and let $j:\text{ker}\,\pi \to M(\epsilon)$ denote this inclusion.

Now we claim that $\text{Hom}(M_{-i}, M(u\epsilon))=0$ for all $i\geq 0$.  Suppose on the contrary that there exist some nonzero $f: M_{-i}\to M(u\epsilon)$ for some $i\geq 0$.  Then $\pi f : M_{-i} \to M(\epsilon)$ must be zero as $M_{-i}$ for $i>0$ lies to the left of $M(\epsilon)$ on a maximal green sequence.  For $i=0$ the composition $\pi f=0$ because $(\epsilon)=M_0$ is a brick. 
Therefore, $\pi f=0$ which means that $f$ factors through  $\text{ker}\,\pi$.  Thus, there exists $g:M_{-i}\to \text{ker}\,\pi$ such that $\iota g=f$.  Then the composition $jg: M_{-i}\to M(\epsilon)$ must again be zero by the same reasoning as before.  Because $j$ is injective, $g=0$ which implies $f=0$, as desired. This shows the claim that $\text{Hom}(M_{-i}, M(u\epsilon))=0$ for all $i\geq 0$.


If $M(u\epsilon)$ where also part of $\mathcal{M}$ then it would have to lie to the right of $M(\epsilon)$ because $M(u\epsilon)$ surjects onto $M(\epsilon)$.  This contradicts the assumption of $M(\epsilon)$ being rightmost in $\mathcal{M}$.  Therefore, $M(u\epsilon)$ cannot lie on this maximal green sequence.  Because $M(u\epsilon)$ is a brick, if we add $M(u\epsilon)$ to $\mathcal{M}$ directly to the right of $M(\epsilon)$ then the resulting sequence cannot remain forward hom-orthogonal.

Since $\text{Hom}(M_{-i}, M(u\epsilon))=0$ for all $i\geq 0$, there must exist $M_{i_1}$ for $i_1>0$ such that we have $\text{Hom}(M(u\epsilon), M_{i_1})\not=0$. Choose $i_1$ to be maximal and some nonzero $f_1: M(u\epsilon) \to M_{i_1}$ such that $\text{im}\,f_1$ is indecomposable. 
Observe that $\text{Hom}(M(\epsilon), M_{i_1})=0$. Therefore, $\text{im}\,f_1$ is not a quotient of $M(\epsilon)$. This implies that $\text{im}\,f_1 = M(v''u^{k-1}vv')$ for some strings $v'', v'$ of nonzero length.
That is, $\text{im}\,f_1$ must be a quotient of $M(u\epsilon)$ but not a quotient of $M(\epsilon)$.  Since $\text{im}\,f_1$ is a quotient of $M(w, \lambda, {k+2})$ and a submodule of $M_{i_1}$, Lemma~\ref{lemma_sub-quo} implies that $v''u^{k-1}vv'$ is a maximal $w$-substring of the string associated to $M_{i_1}$.  Since $M_{i_1}$ is a brick and $\text{im}\,f_1$ is a quotient of a brick $M(u\epsilon)$ condition (B2) of Lemma~\ref{lemmaBB} is satisfied and we conclude that $\text{im}\,f_1$ is also a brick.

Observe that there are no nonzero morphisms $M_{i}\to \text{im}\,f_1$  for any module $M_i$ in $\mathcal{M}$ with $i\leq i_1$.  Otherwise composing such map with the inclusion $\text{im}\,f_1 \to M_{i_1}$ results in a contradiction to $\mathcal{M}$ being a maximal green sequence.  There are is also no nonzero map from $\text{im}\,f_1$ to any module in $\mathcal{M}$ to the right of $M_{i_1}$, because composing such map with the surjection $M(u\epsilon)\to \text{im}\,f_1$ yields a map from $M(u\epsilon)$ to a module to the right of $M_{i_1}$ contradicting the maximality of $i_1$.  This shows that we can place $\text{im}\,f_1$ directly to the right of $M_{i_1}$ in $\mathcal{M}$ and preserve the properties of $\mathcal{M}$ being a forward hom-orthogonal sequence.   This yields a contradiction to $\mathcal{M}$ being a complete forward hom-orthogonal sequence, unless $\text{im}\,f_1 = M_{i_1}$.  

For $k>2$, let $u = z v''$ for some string $z$.  Let $u_2=v''z$ and then $u \sim u_2$, and $v''u^{k-1}vv' = (v''z)^rv_2=u_2^rv_2$ with $r \ge k-1 \ge 2$ 
and $v_2$ a proper substring of $u_2$. Therefore, $\text{im}\,f_1$ is a string module given by a string that satisfies the conditions in the statement of the lemma. This contradicts the fact that $M(\epsilon)$ is rightmost in $\mathcal{M}$ of this particular form.  Thus, we conclude $k=2$ and $\text{im}\,f_1 = M(u_2v_2)=M_{i_1}$ where $\ell(u_2v_2)<\ell(u^2)$ because $M(\epsilon)$ is the rightmost module in $\mathcal{M}$ supported on a square of a minimal band. 


Since $M_{i_1}$ is a quotient of a brick $M(u\epsilon)$ supported on $u_2^2$, condition (C2) of Lemma~\ref{lemmaCC} is satisfied, and we conclude that $M(u_2^{2}v_2)$ is also a brick.  Moreover, $M(u_2^{2}v_2)$ maps surjectively onto $M_{i_1}= M(u_2v_2)$ where the kernel of this surjection is a submodule of $M_{i_1}$. 
By the same reasoning as before there are no nonzero maps $M_{i}\to M(u_2^2v_2)$ for any module $M_i$ in $\mathcal{M}$ with $i\leq i_1$.  Moreover, $M(u_2^2v_2)$ cannot lie on $\mathcal{M}$ because it would lie to the right of $M_{i_1}$ and hence to the right of $M(\epsilon)$.  This means that if we add $M(u_2^2v_2)$ to $\mathcal{M}$ directly to the right of $M_{i_1}$ then the resulting sequence cannot remain forward hom-orthogonal.  Thus, there exists a nonzero map $f_2: M(u_2^2v_2)\to M_{i_2}$ for some $i_2>i_1$ and we pick $i_2$ to be maximal.  Similarly to above, we conclude $\text{im}\,f_2=M_{i_2}=M(u_3v_3)$ where $u_3\sim u_2$ and $v_3$ is a substring of $u_3$.  Then again we construct $M(u_3^2v_3)$ which is a brick and has a surjective map onto $M(u_3v_3)$.  We can continue this procedure moving to the right all the time.  Eventually, this process must stop and we obtain a contradiction, because the sequence $\mathcal{M}$ is finite.  This shows that $M(\epsilon)$ cannot lie on a maximal green sequence. 
\end{proof}

The following theorem says that no module supported on a square of a band can appear on a maximal green sequence.

\begin{thm}\label{thm-mgs}
Let $M$ be an indecomposable module over a string algebra $\Lambda$.  If $M$ is a string module that is supported on $w^2$, for some band $w$, then $M$ cannot lie on a maximal green sequence. 
\end{thm}

\begin{proof}
We proceed in a similar manner as in the proof of Lemma~\ref{lemmaw2}.
Suppose $\mathcal{M}$ is a maximal green sequence for $\Lambda$ in which a brick module $M=M(\gamma)$ supported on $w^2$ appears.  Let $\epsilon$ be a maximal $w$-substring of $\gamma$.  By Lemma~\ref{lemma1}, the module $M(\epsilon)$ is a submodule or a quotient of $M(\gamma)$. We suppose the former, and the other case follows similarly.   By possibly replacing $M(\gamma)$ with another module, we can assume that $M(\gamma)$ is a rightmost module in $\mathcal{M}$ supported on a square of some band and such that a maximal substring of $\gamma$ associated to this band gives rise to a submodule of $M(\gamma)$.

The sequence $\mathcal{M}$ may be written as follows
$$\mathcal{M}:\,\,\,  M_{-q}, M_{-q+1}, \dots, M_{-1}, M_0=M, M_1, \dots, M_{p-1}, M_p$$
where $p,q\geq 0$. Note that by definition of a minimal band, if $M$ is supported on a square of a band then it is also supported on $w^2$, where $w$ is minimal.  Therefore, we may suppose that $w$ is minimal. 

Condition (B1) of Lemma~\ref{lemmaBB} holds and we conclude that $M(\epsilon)$ is a brick.   Let $\epsilon=u^kv$ where $u\sim w$, $k\geq 2$, and $v$ is a proper substring of $u$.  Moreover, by Lemma~\ref{lemmaw2} we may assume $M\not=M(\epsilon)$, and then by Lemma~\ref{lemma4}  we conclude that $M(\epsilon)$ is a quotient of $M(w, \lambda, k+1)$.

Observe that $\text{Hom}(M_{-i}, M(\epsilon))=0$ for all $i\geq 0$.  Indeed, the case $i=0$ follows because $M$ is a brick, and the case $i>0$ follows because $\text{Hom}(M_{-i}, M_0)=0$ and $M(\epsilon)$ is a submodule of $M=M_0$.   Recall that $M(\epsilon)$ cannot lie on a maximal green sequence by Lemma~\ref{lemmaw2}. Then if we add $M(\epsilon)$ to $\mathcal{M}$ directly to the right of $M$, the resulting sequence cannot be forward hom-orthogonal. Thus, there exists a nonzero map $f_1: M(\epsilon)\to M_{i_1}$ for some $i_1>0$.  Moreover, we may pick $i_1$ to be maximal, and we may assume that $\text{im}\,f_1$ is indecomposable.  Now, $\text{im}\,f_1=M(\gamma_1)$ for some string $\gamma_1$ is a quotient of $M(\epsilon)$. For it not to induce a nonzero map $M\to M_{i_1}$, the string $\gamma_1$ must appear at an endpoint of $\epsilon$ that is not an endpoint of $\gamma$ and $\gamma_1$ must be supported on $u^{k-1}$.  Then $\gamma_1$ equals $u^{k-1}vv'$ or $u'u^{k-1}v$ for some strings $u',v'$.  In either case $\text{im}\,f_1$ is supported on $w^{k-1}$ where $w$ is a band. In addition, $\text{im}\,f_1$ is a quotient of $M(w, \lambda, k)$, because $M(\epsilon)$ is a quotient of $M(w,\lambda, k+1)$.  By Lemma~\ref{lemma_sub-quo} the string $\gamma_1$ is a maximal $w$-substring of $\gamma_2$ where $M_{i_1}=M(\gamma_2)$.   Then by condition (B2) of Lemma~\ref{lemmaBB} we conclude that $\text{im}\,f_1$ is a brick.  By (C2) of Lemma~\ref{lemmaCC} we also have that $M(\gamma_1')$, where $\gamma'_1$ is obtained from $\gamma_1$ by replacing $u^{k-1}$ with $u^k$, is also a brick.  Moreover, $M(\gamma_1')$ surjects onto $\text{im}\,f_1$ with kernel being a submodule of $\text{im}\,f_1$.  This implies that there are no nonzero maps $M_i\to M(\gamma_1')$ for $i\leq i_1$ in the sequence $\mathcal{M}$, because $\text{Hom}(M_i, M_{i_1})=0$ for $i<i_1$ and $M_{i_1}$ is a brick.  Moreover, by Lemma~\ref{lemmaw2} the module $M(\gamma_1')$ cannot lie on a maximal green sequence.  

Hence, there exists some map $f_2: M(\gamma_1') \to M_{i_2}$ for some $i_2>i_1$ and we choose $i_2$ to be maximal.  Again we can continue the same argument replacing $f_1$ with $f_2$, and proceeding in the same way until we obtain a contradiction.  The process must stop as $\mathcal{M}$ contains only finitely many modules.  This proves the theorem. 
\end{proof}

Theorem~\ref{thm-mgs} can be used in constructing maximal green sequences as follows.  Generally, one starts with some initial chain of torsion classes, or equivalently an initial segment of a complete FHO sequence of bricks, and obtains the next element in the chain via mutation.  If this new element corresponds to a string module supported on a square of a band, then it follows that such initial segment cannot be completed to a maximal green sequence.  In this case, one needs to go back and choose a different mutation in the hope of obtaining a maximal green sequence.  This result has similar flavor to the one obtained in \cite[Theorem 1]{brustle2017semi} for maximal green sequences of path algebras coming from acyclic quivers.  It says that if there are multiple arrows from vertex $i$ to vertex $j$  then mutating at vertex $j$ before vertex $i$ can never be completed to a maximal green sequence. 

We remark that the bound $k=2$ in Theorem~\ref{thm-mgs} is sharp.  Below we provide an example of a maximal green sequence containing a module $M$ supported $w$ where $w$ is a band.   

\begin{ex}
Let $\Lambda$ be given by the following quiver with relations $\alpha_1\beta_1=\alpha_2\beta_2=0$. 
$$
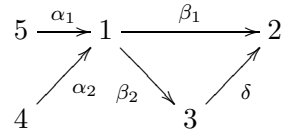
\xymatrix@C=20pt@R=20pt{
5\ar[r]^{\alpha_1} &1\ar[rr]^{\beta_1} \ar[dr]_{\beta_2} &&2\\
4\ar[ur]_{\alpha_2}&&3\ar[ur]_{\delta}}$$

The sequence $\mathcal{M}$ given below is a maximal green sequence for $\Lambda$, and note that the module $\begin{smallmatrix}5\;\;\;\\1\;\;4\\3\;1\\2\end{smallmatrix}=M(\alpha_1\beta_2\delta\beta_1^{-1}\alpha_2^{-1})$ is supported on a band $\beta_2\delta\beta_1^{-1}$.  Here we represent modules via their composition factors. 

$$\mathcal{M}: \,\,\, 
{\begin{smallmatrix}4\end{smallmatrix}},\hspace{.2cm}
{\begin{smallmatrix}5\end{smallmatrix}},\hspace{.2cm}
{\begin{smallmatrix}5\;4\\1\end{smallmatrix}},\hspace{.2cm}
{\begin{smallmatrix}4\\1\end{smallmatrix}},\hspace{.2cm}
{\begin{smallmatrix}5\\1\end{smallmatrix}},\hspace{.2cm}
{\begin{smallmatrix}5\\1\\3\end{smallmatrix}},\hspace{.2cm}
{\begin{smallmatrix}5\;\;\;\\1\;\;4\\3\;1\\2\end{smallmatrix}},\hspace{.2cm}
{\begin{smallmatrix}5\\1\\3\\2\end{smallmatrix}},\hspace{.2cm}
{\begin{smallmatrix}4\\1\\2\end{smallmatrix}},\hspace{.2cm}
{\begin{smallmatrix}1\end{smallmatrix}},\hspace{.2cm}
{\begin{smallmatrix}1\\2\end{smallmatrix}},\hspace{.2cm}
{\begin{smallmatrix}2\end{smallmatrix}},\hspace{.2cm}
 {\begin{smallmatrix}1\\3\end{smallmatrix}},\hspace{.2cm}
 {\begin{smallmatrix}3\end{smallmatrix}}
$$

\end{ex}

Observe that Theorem~\ref{thm-mgs} prohibits infinite families of modules from lying on a maximal green sequence.  However, for general string algebras there are still infinitely many other brick modules that can be supported on $w$ and not on $w^2$ where $w$ is a band.  This allows for the non-exitense of a maximal green sequence or even the existence of infinitely many maximal green sequences for a given string algebra.  In the following special case of a domestic string algebra, we show that the latter scenario is not possible.

First we recall a few definitions.  Let $\Lambda=KQ/I$ be a finite dimensional string algebra.  Its simple module at vertex $a\in Q_0$ is denoted by $S(a)$.  Thus, $S(a)=M(e_a)$ for a string $e_a$ of length zero.  A module is said to be {\it semisimple} if it is a (finite) direct sum of simple modules.  Let $M\in \Lambda$-mod, then the {\it top} (resp., {\it socle}) of $M$ is the largest semisimple quotient (resp., submodule) of $M$.

\begin{cor}\label{dom_str_cor}
If $\Lambda$ is a domestic string algebra, then it admits at most finitely many maximal green sequences. 
\end{cor}

\begin{proof}
Suppose on the contrary that $\Lambda$ is domestic and admits infinity many maximal green sequences.  By Lemma~\ref{lemma-bands}, no band module can lie on a maximal green sequence.  Then there exist infinitely many brick string modules $M$ that lie on maximal green sequences.  In particular, the dimension of these modules grows arbitrarily large. Because $\Lambda$ is finite dimensional, there are no uniserial modules over $\Lambda$ of arbitrarily large length.  Since there are infinitely many such brick modules, there must exist $M$ that lies on a maximal green sequence and that contains a sufficiently large number of copies of some simple module $S(a)$ in its socle such that the same configuration $^{^{\alpha}\,{\searrow} }a \,^{{\swarrow}^{\beta}}$ appears at least three times in the diagram for $M$.   However, the substring of $M$ in between two such adjacent configurations yields a band.  Since $\Lambda$ is domestic, \cite[Corollary 1]{ringel1995some} says that there is at most one band $w$ in $\Lambda$ up to equivalence with a given simple module in its socle.    This implies that $M$ is supported on $w^2$ for some band $w$. By Theorem~\ref{thm-mgs}, this is not possible. 
\end{proof}

The above corollary says that if a domestic string algebra admits a maximal green sequence then it admits finitely many of them.  However, the next example shows that not all domestic string algebras admit a maximal green sequence.  We also make the following observation regarding simple modules. 

\begin{remark}\label{rem:simples}
A maximal green sequence $\mathcal{M}$ for $\Lambda$ contains all simple $\Lambda$-modules.  This follows from the discussion in Section~\ref{sec:MGS} and the beginning of Section~\ref{sec:bricks}, as each module $M_i$ in $\mathcal{M}$ comes from a covering relation $\mathcal{T}_i \lessdot \mathcal{T}_{i+1}$ in $\textsf{tors}\Lambda$  where $M_i$ is the unique smallest brick in $\mathcal{T}_{i+1} \setminus \mathcal{T}_{i}$.
\end{remark} 

\begin{ex}\label{ex1}
Consider a string algebra $\Lambda$ given by the following quiver with relations $\text{rad}^3 \,\Lambda=0$.  

$$
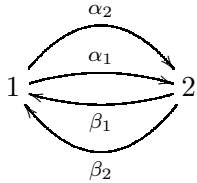
\xymatrix{ 1 \ar@/^/[rr]^{\alpha_1} \ar@/^2pc/[rr]^{\alpha_2}&& 2 \ar@/^/[ll]^{\beta_1} \ar@/^2pc/[ll]^{\beta_2}}$$

Observe that $\Lambda$ is domestic, because there are only two bands $\alpha_1\alpha_2^{-1}$ and $\beta_1\beta_2^{-1}$.  However, it does not admit a maximal green sequence.   
By Remark~\ref{rem:simples} a maximal green sequence for $\Lambda$ contains its simple modules $S(1)$ and $S(2)$.  If $S(1)$ appears before $S(2)$ then there are infinitely many brick modules 
that we need to place in between $S(1)$ and $S(2)$ such that the resulting sequence is forward hom-orthogonal.
$$S(1), \dots ,\begin{smallmatrix}1\,1\,1\\2\,2\,2\,2\,2\end{smallmatrix}, \begin{smallmatrix}1\,1\\2\,2\,2\end{smallmatrix}, \begin{smallmatrix}1\\2\,2\end{smallmatrix}, S(2)$$
This contradicts the definition of $\mathcal{M}$ being finite.  Similarly, if $S(2)$ appears before $S(1)$ then there are infinitely many modules that need be placed between them.   
$$S(2), \dots, \begin{smallmatrix}2\,2\,2\\1\,1\,1\,1\,1\end{smallmatrix}, \begin{smallmatrix}2\,2\\1\,1\,1\end{smallmatrix}, \begin{smallmatrix}2\\1\,1\end{smallmatrix} , S(1)$$
This means that no set of simple modules can be completed to a maximal green sequence.  Next, we will use this idea to prove the existence of maximal green sequences for certain types of string algebras.


Alternatively, we can use the notion of torsion classes to see that there are no maximal green sequences in this particular example.  Recall from Section~\ref{sec:MGS} that a maximal green sequence comes from a finite maximal chain in the lattice of torsion classes of  $\Lambda$-mod.  Such a chain starts at the largest element $\Lambda$-mod and terminates at the zero torsion class.  Below we show the portion of  $\textsf{tors}\Lambda$ connected to $\Lambda$-mod. It shows that there are two infinite chains starting in $\Lambda$-mod, and in particular, there are no finite maximal chains starting in $\Lambda$-mod that terminate at the zero torsion class.


$$
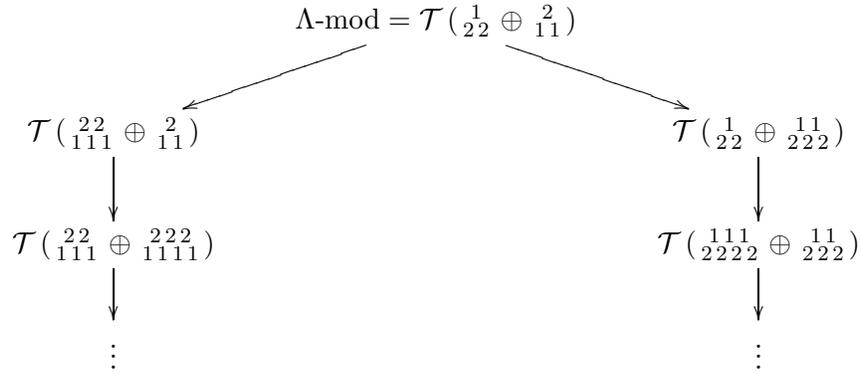
\xymatrix{&{  \Lambda \text{-mod}=\mathcal{T} \, (\begin{smallmatrix}1\\2\,2\end{smallmatrix}} \oplus {\begin{smallmatrix} 2\\1\,1\end{smallmatrix})}\ar[dl]\ar[dr] \\
\mathcal{T}\, ({\begin{smallmatrix}2\,2\\1\,1\,1\end{smallmatrix}}\oplus {\begin{smallmatrix}2\\1\,1\end{smallmatrix}}) \ar[d]&& \mathcal{T}\, ({\begin{smallmatrix}1\\2\,2\end{smallmatrix}}\oplus  {\begin{smallmatrix}1\,1\\2\,2\,2\end{smallmatrix}})\ar[d]\\
\mathcal{T}\, ({\begin{smallmatrix}2\,2\\1\,1\,1\end{smallmatrix}}\oplus {\begin{smallmatrix}2\,2\,2\\1\,1\,1\,1\end{smallmatrix}}) \ar[d]&& \mathcal{T}\, ({\begin{smallmatrix}1\,1\,1\\2\,2\,2\,2\end{smallmatrix}}\oplus {\begin{smallmatrix}1\,1\\2\,2\,2\end{smallmatrix}})\ar[d]\\
\vdots&&\vdots}$$
\end{ex}

The next set of results are aimed at identifying certain properties of string algebras that ensure the existence of a maximal green sequence.

\begin{thm}\label{thm_simples}
Let $\Lambda$ be a string algebra such that no simple module appears in the top of some band module and in the socle of another band module, then $\Lambda$ admits a maximal green sequence. 
\end{thm}

\begin{proof}
If $\Lambda$ is of finite representation type, then there are finitely many nonisomorphic modules so $\Lambda$ admits a maximal green sequence.  If $\Lambda$ is of infinite representation type, then there are band modules in $\text{mod}\,\Lambda$.  By Remark~\ref{rem:simples}, any maximal green sequence contains all simple $\Lambda$-modules.  Let $S(a_1), \dots, S(a_k)$ be a collection of indecomposable simple $\Lambda$-modules such that each $S(a_i)$ appears in a socle of some band module.  Let $S(b_1), \dots, S(b_t)$ be the remaining simple $\Lambda$-modules.  Now consider the following sequence of modules $\mathcal{X}$.
$$\mathcal{X}: \,\, S(a_1), \dots, S(a_k), S(b_1), \dots, S(b_t)$$
Observe that $\mathcal{X}$ is a sequence of bricks with no nonzero maps between distinct elements.  We want to show that $\mathcal{X}$ can be completed to a maximal green sequence. 

Suppose on the contrary that there exists an infinite weakly forward hom-orthogonal sequence of bricks $\mathcal{X}'$ that contains $\mathcal{X}$ as a subsequence.  Because $\mathcal{X}'$ is infinite while $\Lambda$ is finite dimensional, it follows that $\mathcal{X}'$ contains modules of arbitrarily large dimension.  In particular, there exists a module $M(\gamma)$ in $\mathcal{X}'$ such that the following configuration $^{^{\alpha}\,{\searrow} }a \,^{{\swarrow}^{\beta}}$ appearing at least twice in $\gamma$ for some vertex $a$.  Then the substring of $\gamma$ appearing between the two copies of $^{^{\alpha}\,{\searrow} } a \,^{{\swarrow}^{\beta}}$ corresponds to a a band module with $S(a)$ in the socle.   Thus, $a=a_i$ for some $i\in \{1, \dots, k\}$.  Let $S(b)$ be a summand of the top of this band.  Because $M(\gamma)$ is a brick $b\not= a_i$, and by assumption on $\Lambda$ we have $b=b_j$ for some $j\in \{1, \dots, t\}$.   In $\mathcal{X}'$ we have one of the following situations, where $M(\gamma)$ appears between $S(a_i)$ and $S(b_j)$, or after these modules, or before them. 

$$\dots, S(a_i), \dots, M(\gamma), \dots, S(b_j), \dots \hspace{1cm}\text{or} \hspace{1cm} \dots, S(a_i), \dots, S(b_j),\dots, M(\gamma), \dots $$

$$\text{or} \hspace{1cm}  \dots, S(\gamma), \dots, M(a_i), \dots, S(b_j), \dots $$

However, none of these are possible.  There are nonzero maps from left to right, as $S(b_j)$ is in the top of  $M(\gamma)$ so there is a surjective map $M(\gamma) \twoheadrightarrow S(b_j)$ and $S(a_i)$ is in the socle of $M(\gamma)$ so there is an injective map $S(a_i) \hookrightarrow M(\gamma)$.  This yields a contradiction. 
\end{proof}

As an immediate consequence we obtain the following result. 

\begin{cor}\label{deg_3_cor}
If $\Lambda$ is a string algebra such that at each vertex there are at most three arrows then $\Lambda$ admits a maximal green sequence. 
\end{cor}

\begin{proof}
If $\Lambda$ is as above then no simple module can appear both in the top of some band module and in a socle of another band module.  The result follows from Theorem~\ref{thm_simples}.
\end{proof}


Next, in the case of a domestic gentle algebra we can always construct a maximal green sequence from a particular ordering of its simple modules.

\begin{thm}\label{dom_gentle_thm}
A domestic gentle algebra admits a maximal green sequence. 
\end{thm}

\begin{proof}
Let $\Lambda$ be a domestic gentle algebra.  If it is of finite representation type, then the conclusion follows.  Otherwise, suppose $\Lambda$ is of infinite representation type, and let $W$ be a set of bands in $\Lambda$ that do not have the same simple module appearing both in the top and in the socle of the same band.  Since $\Lambda$ is domestic, we conclude that $W$ is finite. 

Now, we construct a particular sequence of simple $\Lambda$-modules using the following procedure.   Let $S(a_1)$ be a summand of the top of some band module $M(w_1, \lambda, 1)$ where $w_1\in W$.  Then there are two distinct paths $a_1\to \dots \to a_2$ and $a_1\to \dots \to a_2'$ in $w_1$ starting in $a_1$ and ending in $a_2, a_2'$ such that $S(a_2), S(a_2')$ appear in the socle of $M(w_1, \lambda, 1)$.   Choose one of the two endpoints, say $a_2$ and obtain the sequence of simple modules $S(a_2), S(a_1)$.  Note that $a_1\not=a_2$ because $w_1\in W$.    If $S(a_2)$ does not appear in the top of any other band module coming from $W$, then we stop.  

Otherwise, there exists a band $w_2\in W$ such that $S(a_2)$ is a summand of the top of $M(w_2, \lambda, 1)$.  There are two distinct paths $a_2\to \dots \to a_3$ and $a_2\to \dots \to a_3'$ in $w_2$ starting in $a_2$ and ending in $a_3, a_3'$ such that $S(a_3), S(a_3')$ appear in the socle of $M(w_2, \lambda, 1)$.   Because $\Lambda$ is gentle, exactly one of these paths precomposed with $a_1\to \dots \to a_2$ equals zero.  Without loss of generality suppose that the path from $a_1$ to $a_2$ to $a_3$ is nonzero, and obtain a sequence of simple modules $S(a_3), S(a_2), S(a_1)$.    If $S(a_3)$ does not appear in the top of any other band module coming from $W$, then we stop.  

Otherwise, we can keep going in this way, until we obtain a sequence $S(a_k), \dots, S(a_1)$ where $S(a_k)$ is not in the top of any band coming from $W$.  Observe that this process must terminate as the algebra $\Lambda$ is finite dimensional and any path starting in $a_1$ must eventually stop.   

Now, we claim that the simple modules along this sequence are distinct.   Suppose on the contrary that some simple $S(a_i)$ for $i\in \{1, \dots, k\}$ appears twice.  Then, there is a nonzero path in $\Lambda$ given by
$$\gamma: a_i\to \dots \to a_{i+1}\to \dots \to a_j \to \dots \to a_{i}=a_{j+1},$$ where $S(a_{j}), \ldots, S(a_{i+1})$ are the simples appearing between the two copies of $S(a_i)$. That the set $\{S(a_j), \ldots, S(a_{i+1})\}$ is nonempty follows the fact that there are no bands in $W$ have the same simple module appearing in both its top and its socle.  

The band $w_j$ consists of the string 
$$a'_{j+1}\leftarrow \dots \leftarrow a_j\to \dots \to a_{j+1}=a_i \leftarrow \dots \to a'_{j+1}.$$
  
Let $\delta_{j+1}$ denote the substring of $w_j$ given by $a_i \leftarrow \dots \to a_{j+1}'$.
Similarly, in the band $w_{j-1}$ there exists a string $\delta_{j}$ given by $a_j \leftarrow \dots \to a_{j}'$.
Now consider the string $\delta_{j}$ composed with the directed path $a_j\to \dots \to a_{j+1}'$ resulting in the expression 
$$\rho_{j}:  a_{j+1}'\leftarrow\dots \leftarrow a_j \leftarrow \dots \to a_{j}'.$$   
We claim that it is a string, because $\Lambda$ is gentle.  Indeed, $\rho_{j}$ is a path through $a_j$ and we already know that $a_{j-1}\to \dots \to a_{j}\to \dots \to a_{j+1}$ is another nonzero path through the same vertex. 
Now, define the remaining 
$$\delta_r: a_{r} \leftarrow \dots \to a_{r}'$$  
for $r\in \{j+1, \dots, i\}$ and 
$$\rho_r:a_{r+1}'\leftarrow\dots \leftarrow a_{r} \leftarrow \dots \to a_{r}'$$
for $r\in \{j, \dots, i+2\}$ in the same way.   Finally, we obtain the following band given by the compositions of paths. 
$$\gamma   \delta_{j+1}   \rho_{j}     \rho_{j-1}  \dots   \rho_{i+1}   \rho_i  (a_{i+1}' \leftarrow \dots \leftarrow a_i)$$


This gives a band module with $S(a_i)$ in its top and it is different from the band $w_i$, which also has $S(a_i)$ in its top. As mentioned above, \cite[Corollary 1]{ringel1995some} says that there is at most one band $w$ in $\Lambda$ up to equivalence with a given simple module in its socle. This is equivalent to the statement that there is at most one band $w$ in $\Lambda$ up to equivalence with a given simple module in its top. We obtain a contradiction. So there are no repeated module in the sequence $S(a_k),\ldots, S(a_1)$. 


Now, given such a sequence of distinct simple modules $S(a_k), \dots, S(a_1)$ we can dually start with $S(a_{k}), S(a_{k-1})$ and the band $w_{k-1}$ and move backwards until we reach $S(a_1)$.  If $S(a_1)$ does not appear in the socle of some other band module $M(w_{-1}, \lambda, 1)$ then we stop.  Otherwise, we obtain a simple $S(a_{-1})$, where $S(a_{-1})$ is in the top of $M(w_{-1}, \lambda, 1)$ and $S(a_1)$ is in its socle.  Moreover, there is a directed path $a_{-1}\to \dots \to a_1$ in $w_{-1}$ such that the composition with $a_1\to \dots \to a_k$ is nonzero.  Hence, we can continue in this way until the process terminates, and, by the same argument as above, we obtain a sequence of distinct simple modules 
$$\mathcal{X}_1: \,\, S(a_k), \dots, S(a_1), S({a_{-1}}), \dots, S(a_{-t+1}), S(a_{-t})$$
where $S(a_k)$ does not appear in the top of some band coming from $W$ and $S(a_{-t})$ does not appear in the socle some band coming from $W$.   

Let $W_1\subset W$ be a collection of bands $w$ such that $S(a_i)$ for some $i\in \{-t, \dots, k\}$ appears in the top or the socle of $M(w, \lambda, 1)$.  Because $\Lambda$ is domestic $W_1 = \{w_{-t},\dots, w_{-1}, w_1, \dots, w_{k-1}\}$.  Next, we show that no string module $M$ supported on a square of a band in $W_1$ can be added to $\mathcal{X}_1$ such that the resulting sequence is weakly forward hom-orthogonal. Suppose $M$ is supported on $w^2_i\in W_1$ for some $i\in \{-t, \dots, k-1\}$.  Then $M$ has $S(a_i)$ in its top and $S(a_{i+1})$ in its socle.  Then, $M$ cannot be placed to the left of $S(a_i)$ or to the right of $S(a_{i+1})$ in $\mathcal{X}_1$.  This shows that no such $M$ can be added to $\mathcal{X}_1$, as desired. 

Now consider $W^2 = W\setminus W_1$.  Let $S(b_1)$ be in the top of some band module coming from $W^2$.  Observe that $S(b_1)$ does not appear in $\mathcal{X}_1$, because there is a unique band in $\Lambda$ with this property.  Now, construct a sequence of simple modules $\mathcal{X}_2$ starting from $S(b_1)$ in a similar way as above.  By the same reasoning we conclude that no string module $M$ supported on a square of a band in $W^2$ can be added to $\mathcal{X}_2$ such that the resulting sequence has no nonzero maps from right to left.   Then $W_2 \subset W^2$ consists of all bands that have $S(b_j)$ in its top or socle such that $S(b_j)$ lies in $\mathcal{X}_2$.   

Next, consider $W^3=W^2\setminus W_2$ and keep going in this way.  Eventually, this process stops because $W$ is a finite set, and we obtain sequences of distinct simple modules that we can combined to create a single sequence $\mathcal{X}$ of distinct simple modules as follows. 
$$\mathcal{X}:\,\,\mathcal{X}_1, \mathcal{X}_2, \dots, \mathcal{X}_p$$  
By construction no string module supported on a square of a band in $W$ can be added to $\mathcal{X}$ such that the resulting sequence is weakly forward hom-orthogonal.   All other string modules supported on a square of a band not in $W$ are not bricks.  Therefore, since $\Lambda$ is domestic, we conclude that there are at most finitely many brick modules that can be added to $\mathcal{X}$ so that the resulting sequence is forward hom-orthogonal.  Using Theorem~\ref{MGS_thm}, this shows that $\mathcal{X}$ can be completed to a maximal green sequence and proves the theorem. 
\end{proof}

It is not clear whether the condition on the algebra being domestic can be removed from the statement of the previous theorem.  The algebra in Example~\ref{ex1} that does not admit a maximal green sequence is domestic but not gentle.  The question remains whether there are gentle algebras that do not admit any maximal green sequences.




The following corollary summarizes some of the results in this section. 

\begin{cor}
Let $\Lambda=KQ/I$ be a domestic string algebra satisfying one of the following conditions.
\begin{itemize}
\item No simple $\Lambda$-module appears in the top of some band module and in the socle of another band module.
\item There are at most three arrows incident to every vertex of $Q$. 
\item $\Lambda$ is gentle. 
\end{itemize}
Then the set of maximal green sequences for $\Lambda$ is finite and nonempty.  

\end{cor}

We remark that all of the statements above on the existence of maximal green sequences are obtained by constructing a particular sequence of simple modules that can always be completed to a finite FHO sequence regardless of what other modules we add to it.  However, Example~\ref{ex:simples} shows that the information contained only in a sequence of simple modules is generally not enough.  Sometimes a given sequence of simple modules completes to a finite FHO sequence and other times it does not. Hence, it might be possible to obtain further results on the existence of maximal green sequences by starting with weakly FHO sequences made up of other types of module in addition to just the simple ones.   For instance the next step would be to consider modules of dimension at most two, as they carry the additional information about the structure of the arrows in the quiver.

\begin{ex}\label{ex:simples}
Let $\Lambda$ be the path algebra of the following quiver without any relations.  This is an algebra of affine type $\tilde{A}_{1,2}$ and contains a single band $w=\beta_1\beta_2\alpha^{-1}$. 

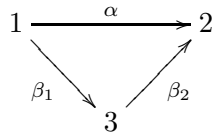
\[\xymatrix{1\ar[rr]^{\alpha}\ar[dr]_{\beta_1}&&2\\ & 3\ar[ur]_{\beta_2}}
\]

Consider a sequence of simple $\Lambda$-modules $\mathcal{X}: S(1), S(2), S(3)$.  Then if we place the module $\begin{smallmatrix}1\\3\end{smallmatrix} = M(\beta_1)$ in between $S(2)$ and $S(3)$ then there are only finitely many modules that can be added to the resulting sequence.  In particular, regardless of what other bricks we add, it can always be completed to a maximal green sequence, for example as follows. 

\[S(1), \begin{smallmatrix}1\\2\end{smallmatrix}, S(2), \begin{smallmatrix}1\\3\end{smallmatrix},  S(3)\]

On the other hand, if we place $\begin{smallmatrix}1\\3\end{smallmatrix}$ in between $S(1)$ and $S(2)$, then regardless of what other bricks we add to this sequence, we still obtain the following weekly FHO sequence that contains a module $M(w^2)$ supported on a square of a band.  By Theorem~\ref{thm-mgs} it cannot be completed to a maximal green sequence. 

 \[S(1), \begin{smallmatrix}1\\3\end{smallmatrix}, \hspace{-.15in} \begin{smallmatrix}1\,\,\,\,\,\,\,\,\,\,\\\,3\,\,\,1\\ \,\,\,\,\,\,\,\,\,\,\,2\,\,\,3\,\,\,1\\\,\,\,\,\,\,\,\,\,\,\,\,\,\,\,\,\,\,2\end{smallmatrix}, S(2),  S(3)\]
 
In particular, we can keep placing modules $M(w^3), M(w^4), \dots, $ supported on higher and higher powers of the band to the right of $M(w^2)$.
 \end{ex}

\bibliography{bib_mgs_strings.bib}
\bibliographystyle{plain}

\end{document}